\documentclass[11pt,reqno]{amsart}
\usepackage{amsmath, amssymb, graphicx, caption}
\usepackage{hyperref}
\usepackage{fullpage}
\usepackage{color}
\usepackage{bm}
\usepackage{natbib}
\usepackage{graphicx}
\usepackage{comment}
\usepackage{fullpage}
\usepackage{subcaption}

\usepackage[normalem]{ulem}

\newtheorem{lemma}{Lemma}
\newtheorem{corollary}{Corollary}
\newtheorem{proposition}{Proposition}

\newtheorem{remark}{Remark}
\newtheorem{assumption}{Assumption}

\counterwithin{figure}{section} 
\counterwithin{definition}{section} 
\counterwithin{theorem}{section} 
\counterwithin{lemma}{section} 
\counterwithin{corollary}{section} 
\counterwithin{proposition}{section} 
\counterwithin{example}{section} 
\counterwithin{remark}{section} 
\counterwithin{assumption}{section} 
\counterwithin{exercise}{section} 

\newcommand{\argmin}{\mathop{\rm arg\,min}\limits}
\allowdisplaybreaks

\begin{document}

\title[]{Random sets from the perspective of metric statistics}
\thanks{The work of D. K. is partially supported by JSPS KAKENHI Grant Numbers 23K12456 and 26K00323} 
\date{\today}

\author[D. Kurisu]{Daisuke Kurisu}
\author[Y. Okamoto]{Yuta Okamoto}
\author[T. Otsu]{Taisuke Otsu}

\address[D. Kurisu]{Faculty of Economics, The University of Tokyo 7-3-1, Hongo, Bunkyo-ku, Tokyo 113-0033, Japan.}
\email{d.kurisu@e.u-tokyo.ac.jp}

\address[Y. Okamoto]{Graduate School of Economics, Kyoto University, Yoshida Honmachi, Sakyo, Kyoto 606-8501, Japan.
}
\email{okamoto.yuuta.57w@st.kyoto-u.ac.jp}

\address[T. Otsu]{Department of Economics, London School of Economics, Houghton Street, London, WC2A 2AE, UK.}
\email{t.otsu@lse.ac.uk}

\begin{abstract}
Since the seminal work by \cite{BeMo08}, the random set theory and related inference methods have been widely applied in partially identified econometric models. Meanwhile, there is an emerging field in statistics for studying random objects in metric spaces, called metric statistics. This paper clarifies a relationship between two fundamental concepts in these literatures, the Aumann expectation and Fr\'echet mean, and presents some applications of metric statistics to econometric problems involving random sets.
\end{abstract}


\maketitle

\section{Introduction}

Since the seminal work by \cite{BeMo08}, the random set theory and related inference methods have been widely applied in partially identified econometric models; see \cite{MoMo18} and \cite{Mo20} for surveys in econometrics, and \cite{Mo05} for a comprehensive overview of random set theory. The random set approach characterizes an identification region of interest by utilizing the \cite{Au65} expectation for set-valued random variables (SVRVs) and constructs its sample analog estimator by using  Minkowski averages of SVRVs. Obviously the key ingredient for partial identification analysis using random sets is the Aumann expectation, which is a natural and convenient extension of the conventional mean for Euclidean random variables to SVRVs.

On the other hand, in the recent statistics literature for analyzing complex data, there is an emerging and rapidly growing field, called {\it metric statistics} (or statistics for random objects); see \cite{DuChMu24} for an overview and references therein. Metric statistics is concerned with complex data situated in a metric space, and popular examples include distributional data, network data, symmetric positive definite matrices, trees, data on Riemannian manifolds, among others. In this literature, a fundamental notion to characterize the population mean for metric space-valued random objects is the \cite{Fr48} mean.

Given these literatures, it is natural to ask whether there is a relationship between these notions of population means, and also whether the methodologies of metric statistics can shed new light on econometric analysis of random sets. This is of course not the first paper addressing these issues. For example, Section 2.3.2 of \cite{Mo05} introduced the Fr\'echet mean and mentioned that it can be applied to SVRVs that situated in the Hausdorff metric space. \cite{LiMoMoPe21} mentioned that their local regression smoother based on the Minkowski average can be interpreted as the sample Fr\'echet mean. Although these discussions are highly insightful, their main focuses are on the conventional random set analysis using the Aumann expectation, and the analysis using the Fr\'echet mean is not pursued. On the other hand, to the best of our knowledge, there is no formal study in the literature of metric statistics on SVRVs. This paper is written to fill this gap.

In particular, this paper makes two contributions. First, we formally study the relationship between the Aumann expectation and Fr\'echet mean, and establish the equivalence of these notions of population means in the space of nonempty compact and convex sets equipped with the $L^2$-metric based on their support functions. A key ingredient to establish such equivalence is an isometric embedding of a general metric space into a Hilbert space. Second, we apply or extend some methodologies of metric statistics to econometric problems that involve random sets. After introducing the global Fr\'echet regression (GFR) by \cite{PeMu19}, we apply the GFR to the projection model of Euclidean covariates, and clarify the relation with the set valued best linear predictor by \cite{BeMo08}. Furthermore, we present some extensions of the GFR for an errors-in-variables model with set-valued outcomes and a missing set-valued data problem.
Moreover, we extend the synthetic control method to accommodate SVRVs. We further discuss how methodologies developed in the metric statistics literature can be applied to econometric problems that have received relatively little attention in the random set literature, including single-index models, endogenous sample selection, and other methods for causal inference.

This paper is organized as follows. Section \ref{sec:main} discusses our main result, equivalence of the Aumann expectation and Fr\'echet mean. Then Section \ref{sec:stat} presents applications of metric statistics for random sets.

\section{Main result}\label{sec:main}

This section discusses the relationship between the Aumann expectation and the Fr\'echet mean of SVRVs. After introducing our basic setup and the Aumann expectation (Section \ref{sub:Aumann}), we introduce a metric space based on the support function (Section \ref{sub:supp}). Then we introduce the Fr\'echet mean and present a key property, isometric embedding (Section \ref{sub:Frechet}). Based on these preparations, Section \ref{sub:equiv} establishes the equivalence of the Aumann expectation and Fr\'echet mean of SVRVs.

\subsection{Setup and Aumann expectation}\label{sub:Aumann}

We follow the notation in \cite{BeMo08} and the standard framework for random sets in \cite{Mo05}, and introduce our basic setup. Let $(\Omega,\mathcal A,\mu)$ be a complete nonatomic probability space, and $K(\mathbb R^d)$ be the collection of all nonempty closed subsets of $\mathbb R^d$. A measurable set-valued map $F:\Omega\to K(\mathbb R^d)$ is called a set-valued random variable (SVRV), where measurability is understood in the sense of Molchanov~(\citeyear{Mo05}, Definition~1.1.1). Let $K_{k}(\mathbb{R}^d)$ be the set of nonempty compact subsets of $\mathbb{R}^d$, and $K_{kc}(\mathbb{R}^d)$ be the set of nonempty compact and convex subsets of $\mathbb{R}^d$. For an SVRV $F$ taking values in $K_k(\mathbb{R}^d)$, the Aumann expectation is defined as follows.

Let $L^1=L^1(\Omega, \mathbb{R}^d)$ be the space of measurable random variables with values in $\mathbb{R}^d$ s.t. $\|\xi\|_1 = \mathbb{E}[\|\xi\|]<\infty$. Let $\mathcal S(F)$ denote the set of all measurable selections of $F$, that is, all measurable random vectors $f$ such that $f(\omega)\in F(\omega)$ for almost all $\omega\in\Omega$, and let $\mathcal S^1(F)=\mathcal S(F)\cap L^1$. The Aumann expectation of an SVRV $F$ taking values in $K_k(\mathbb{R}^d)$ is defined as
\begin{equation}
\mathbb{E}[F] = \left\{ \mathbb{E}[f]: f \in \mathcal{S}^1(F)\right\}.
\end{equation}
This coincides with the Aumann integral of a random set; see Molchanov~(\citeyear{Mo05}, eq.~(1.5)). For a sub-$\sigma$-field $\mathcal G\subset\mathcal A$, let $\mathbb E[F|\mathcal G]$ be the conditional Aumann expectation of $F$ given $\mathcal G$ as defined in Molchanov~(\citeyear{Mo05}, Theorem~1.46). For an $\mathbb R^k$-valued random vector $X$ with support $\mathcal X$,
we write $\mathbb E[F| X]=\mathbb E[F|\sigma(X)]$, where $\sigma(X)$ is the $\sigma$-field generated by $X$. The Aumann expectation is a natural generalization of the conventional mean for Euclidean random variables and plays a fundamental role in random set theory and its econometric and statistical applications; see \cite{Mo05}, \cite{MoMo18}, and \cite{Mo20}.

Hereafter we focus on SVRVs on the set of nonempty compact and convex subsets, $K_{kc}(\mathbb{R}^d)$, and characterize the Aumann expectation $\mathbb{E}[F]$ for $F$ taking values in $K_{kc}(\mathbb{R}^d)$ by using the notion of the Fr\'echet mean for random objects in a metric space, which has been increasingly popular in recent literature on metric statistics; see \cite{DuChMu24} for example.

\subsection{Support function and metric on $K_{kc}(\mathbb{R}^d)$}\label{sub:supp}

A key ingredient to clarify the relationship between the Aumann expectation and Fr\'echet mean is to choose a proper metric on the set of nonempty compact and convex subsets $K_{kc}(\mathbb{R}^d)$. To this end, for a deterministic set $A\in K_{kc}(\mathbb{R}^d)$, we introduce the support function $s(p,A)=\sup_{a \in A}\langle p,a \rangle$ over $p\in \mathbb{S}^{d-1}$, where $\mathbb{S}^{d-1} = \{x \in \mathbb{R}^d: \|x\|=1\}$ is the unit sphere in $\mathbb{R}^d$. For an SVRV $F$ taking values in $K_{kc}(\mathbb{R}^d)$, the Aumann expectation $\mathbb{E}[F]$ is equivalently characterized by its support function, i.e., $s(p,\mathbb{E}[F])=\mathbb{E}[s(p,F)]$ for $p\in\mathbb{S}^{d-1}$.

Let $L^2(\mathbb S^{d-1})$ denote the space of square-integrable real-valued functions on $\mathbb S^{d-1}$ equipped with the norm $\|f\|_{2,\mathbb S^{d-1}} = \left(\int_{\mathbb S^{d-1}} |f(p)|^2\,dp\right)^{1/2}$. For $A, B \in K_{kc}(\mathbb{R}^d)$, define the metric $d_{kc}$ by
\begin{equation}
d_{kc}(A, B)=\|s(\cdot,A)-s(\cdot,B)\|_{2,\mathbb S^{d-1}}.
\end{equation}
If $F$ and $G$ are SVRVs taking values in $K_{kc}(\mathbb{R}^d)$, then $d_{kc}(F,G)$ is a real-valued random variable. This paper employs $d_{kc}$ as a metric on $K_{kc}(\mathbb{R}^d)$ and establishes the equivalence of the Aumann expectation and Fr\'echet mean. 

We close this subsection by presenting some properties related to $d_{kc}$. Let $d_H(\cdot,\cdot)$ be the Hausdorff metric on $K_k(\mathbb{R}^d)$ defined by
\begin{equation}
d_H(A,B) = \max\Bigl\{ \sup_{a\in A}\inf_{b\in B}\|a-b\|, \sup_{b\in B}\inf_{a\in A}\|a-b\| \Bigr\},
\end{equation}
for $A,B\in K_k(\mathbb{R}^d)$, define $K_{kc}^R(\mathbb{R}^d) = \{A\in K_{kc}(\mathbb{R}^d):\sup_{a\in A}\|a\|\le R\}$ for $R>0$, and let $\Psi:K_k(\mathbb{R}^d)\to L^2(\mathbb S^{d-1})$ be the map defined by $\Psi(A)=s(\cdot,A)$. For $A,B\in K_{kc}(\mathbb R^d)$ and $t\in\mathbb R$, let $A \oplus B=\{a+b:a\in A,\ b\in B\}$ be the Minkowski sum and $tA=\{ta:a\in A\}$.

\begin{lemma}\label{lem:d}
\quad
    \item[(i)] $d_{kc}$ is a metric on $K_{kc}(\mathbb{R}^d)$.
    \item[(ii)] Let $\{A_n\}_{n \geq 0} \subset K_{kc}^R(\mathbb{R}^d)$ be a deterministic sequence of compact convex sets. Then $\lim_{n \to \infty}d_{kc}(A_n, A_0) = 0$ if and only if $\lim_{n \to \infty}d_H(A_n, A_0) = 0$.
    \item[(iii)] $K_{kc}^R(\mathbb{R}^d)$ is a bounded closed convex subset of $K_{kc}(\mathbb{R}^d)$ with respect to $d_{kc}$. 
    \item[(iv)] $\Psi(K_{kc}^R(\mathbb{R}^d))$ is a bounded closed convex subset of $L^2(\mathbb{S}^{d-1})$ equipped with the norm $\|\cdot\|_{2,\mathbb S^{d-1}}$.
\end{lemma}
Lemma \ref{lem:d} establishes several topological properties of compact convex sets under the metric $d_{kc}$.
In particular, Lemma \ref{lem:d} (ii) shows that convergence under the Hausdorff metric, which can be defined as the sup-norm distance between support functions on $\mathbb{S}^{d-1}$ (see Lemma \ref{lem:Hd-sf-conv}), is equivalent to convergence under the $L^2$-based distance $d_{kc}$. This equivalence motivates the use of $d_{kc}$, which is analytically more tractable, for statistical analysis of random compact convex sets. In Section \ref{sub:equiv}, we show that the Aumann expectation (and its sample counterpart, the normalized Minkowski sum) can be characterized as the Fr\'echet mean, that is, the minimizer of $\mathbb{E}[d_{kc}^2(\nu,F)]$, the expectation of the squared distance between a data descriptor $\nu \in K_{kc}^R(\mathbb{R}^d)$ and an SVRV $F$ taking values in $K_{kc}^R(\mathbb{R}^d)$ (and its sample version). Finally, Lemma \ref{lem:d} (iii) and (iv) guarantee the existence and uniqueness of the Fr\'echet mean.

\subsection{Fr\'echet mean and isometric embedding}\label{sub:Frechet}

We now introduce the Fr\'echet mean for random objects. Let $\mathcal{X}$ be a subset of $\mathbb{R}^k$, $(\mathcal{M},d)$ be a separable metric space, and $(Y,X) \in \mathcal{M} \times \mathcal{X}$ be a pair of random elements. The Fr\'echet mean $\mathbb{E}_\oplus[Y]$ of $Y$ is defined as a minimizer of the population Fr\'echet function $Q(\nu)=\mathbb{E}[d^2(\nu,Y)]$, that is,  
\begin{equation}
\mathbb{E}_\oplus[Y] \in \mathrm{argmin}_{\nu \in \mathcal{M}}Q(\nu).
\end{equation}
Similarly, the conditional Fr\'echet mean $\mathbb{E}_\oplus[Y|X]$ of $Y$ given $X$ is defined as 
\begin{equation}
\mathbb{E}_\oplus[Y|X] \in  \mathrm{argmin}_{\nu \in \mathcal{M}}Q(\nu|X),\quad Q(\nu|X)=\mathbb{E}[d^2(\nu,Y)|X]. 
\end{equation}
The Fr\'echet mean is a direct generalization of the conventional mean for the Euclidean space toward a general metric space, and its statistical analysis has been increasingly popular in recent literature. The key step to establish the relationship between the Aumann expectation and Fr\'echet mean is to consider an isometric embedding of a general metric space into a Hilbert space. To this end, we impose the following assumptions.

\begin{assumption}\label{ass:1} \quad 
\begin{itemize}
\item[(i)] There exist a Hilbert space $\mathcal{H}$ equipped with an inner product $\langle \cdot, \cdot \rangle_\mathcal{H}$, induced norm $\|\cdot\|_\mathcal{H}$, and a continuous injection $\Psi:\mathcal{M} \to \mathcal{H}$ such that $\Psi: \mathcal{M} \to \Psi(\mathcal{M})$ is an isometry, i.e., $d(\alpha,\beta)=\|\Psi(\alpha)-\Psi(\beta)\|_\mathcal{H}$ for any $\alpha,\beta \in \mathcal{M}$.
\item[(ii)] The set $\Psi(\mathcal{M})$ is a nonempty closed convex set in $\mathcal{H}$.  
\end{itemize}
\end{assumption}
$(K_{kc}^R(\mathbb{R}^d),d_{kc})$ satisfies Assumption \ref{ass:1}. Indeed, setting $\Psi(A) = s(\cdot, A)$ implies that $K_{kc}^R(\mathbb{R}^d)$ admits an isometric embedding into $L^2(\mathbb{S}^{d-1})$. Furthermore, by Lemma \ref{lem:d} (iv), the image $\Psi(K_{kc}^R(\mathbb{R}^d))$ satisfies Assumption \ref{ass:1} (ii). Let $(\mathcal{M},d)$ be a metric space. The metric $d^2$ is of negative type if, for all $n \geq 2$, $\nu_1,\dots,\nu_n \in \mathcal{M}$ and $\alpha_1,\dots, \alpha_n \in \mathbb{R}$ with $\sum_{i=1}^n\alpha_i = 0$, we have $\sum_{i=1}^n\sum_{j=1}^n \alpha_i\alpha_j d^2(\nu_i,\nu_j) \leq 0$. A sufficient condition for Assumption \ref{ass:1} (i) is provided in Proposition 3 in \citet{sejd:13}, which implies that if $d^2$ is of negative type, then this condition is satisfied; see also \citet{scho:38}. Many of the common metric spaces studied in metric statistics are known to satisfy Assumption \ref{ass:1}. Examples include the 2-Wasserstein space for univariate probability distributions, the space of symmetric positive (semi)definite matrices endowed with the Frobenius, power, or log-Euclidean metric, the space of graph Laplacians equipped with the Frobenius metric, and the space of compositional data endowed with the Aitchison metric. See also Appendix C.1 in \cite{KuZhOtMu25b}.

Under these assumptions, we obtain the following characterizations of the Fr\'echet mean. For a random element $Z$ taking values in a Hilbert space $\mathcal{H}$, we define its expectation as the Riesz representation of the linear functional that maps $h \in \mathcal{H}$ to $\mathbb{E}[\langle h,Z \rangle_\mathcal{H}] \in \mathbb{R}$.

\begin{proposition}\label{prp:F-mean-embed}
Suppose that Assumption \ref{ass:1} holds true and $\mathbb{E}[\|\Psi(Y)\|_\mathcal{H}^2]<\infty$. Then the following results hold. 
\begin{itemize}
\item[(i)] The object $\Psi^{-1}(\mathbb{E}[\Psi(Y)])$ is well defined, and $\mathbb{E}_\oplus[Y]= \Psi^{-1}(\mathbb{E}[\Psi(Y)])$.  
\item[(ii)] $\Psi(\mathbb{E}_\oplus[Y]) = \mathbb{E}[\Psi(\mathbb{E}_\oplus[Y|X])]$.
\end{itemize}
\end{proposition}
According to Proposition \ref{prp:F-mean-embed} (i), the Fr\'echet mean of $Y$ is obtained as the pullback, through $\Psi$, of the expectation of $\Psi(Y)$ in the image space $\Psi(K_{kc}^R(\mathbb{R}^d))$. Proposition \ref{prp:F-mean-embed} (ii) can be interpreted as a law of iterated expectation for random objects, which naturally generalizes the corresponding result for Euclidean random variables.

\subsection{Equivalence of Aumann expectation and Fr\'echet mean}\label{sub:equiv}

We now apply the Fr\'echet mean to SVRVs in the metric space $(K_{kc}^R(\mathbb{R}^d),d_{kc})$. Let $(F,X)$ be a pair of random elements taking values in $K_{kc}^R(\mathbb{R}^d) \times \mathcal{X}$. The population Fr\'echet mean $\mathbb{E}_\oplus[F]$ of $F$ with respect to the metric $d_{kc}$ is defined as a minimizer of the population Fr\'echet function $\mathcal{Q}(\nu)=\mathbb{E}[d_{kc}^2(\nu,F)]$, that is,  
\begin{equation}
\mathbb{E}_\oplus[F] \in  \mathrm{argmin}_{\nu \in K_{kc}^R(\mathbb{R}^d)}\mathcal{Q}(\nu).     
\end{equation}
Similarly, the population conditional Fr\'echet mean of $F$ given $X$ is defined as 
\begin{equation}
    \mathbb{E}_\oplus[F|X] \in  \mathrm{argmin}_{\nu \in K_{kc}^R(\mathbb{R}^d)}\mathcal{Q}(\nu|X), \quad \mathcal{Q}(\nu|X) = \mathbb{E}[d_{kc}^2(\nu,F)|X]. 
\end{equation}
The main results of this paper are presented as follows.
\begin{proposition}\label{prp: SVRV-exp-cexp}
For an SVRV $F:\Omega \to K_{kc}^R(\mathbb{R}^d)$, the following results hold true. 
\begin{itemize}
\item[(i)] $\mathbb{E}_\oplus[F]$ and $\mathbb{E}_\oplus[F|X]$ with respect to the metric $d_{kc}$ uniquely exist.
\item[(ii)] $\mathbb{E}_\oplus[F] = \mathbb{E}[F]$.
\item[(iii)] $\mathbb{E}_\oplus[F|X] = \mathbb{E}[F|X]$ and $\mathbb{E}_\oplus[F] = \mathbb{E}_\oplus[\mathbb{E}_\oplus[F|X]]$.   
\end{itemize}
\end{proposition}

Proposition \ref{prp: SVRV-exp-cexp} (i) guarantees uniqueness of the Fr\'echet means $\mathbb{E}_\oplus[F]$ and $\mathbb{E}_\oplus[F|X]$. Proposition \ref{prp: SVRV-exp-cexp} (ii) and (iii) are our main results, equivalence of the Fr\'echet and conditional Fr\'echet means to the Aumann and conditional Aumann expectations, respectively. Furthermore, Proposition \ref{prp: SVRV-exp-cexp} (iii) also provides the law of iterated expectations for the Fr\'echet mean.

An SVRV $F$ taking values in $K_{kc}(\mathbb{R}^d)$ is a metric space-valued random element, which is called random object. See e.g., \cite{MaAl14} and \cite{MaDr21} for a review. By Proposition \ref{prp: SVRV-exp-cexp}, the (conditional) Aumann expectation of $F$ can be defined as the (conditional) Fr\'echet mean of $F$ under the metric $d_{kc}$. From this viewpoint, statistical analysis of SVRVs falls within the scope of metric statistics, a field that has seen remarkable development in recent years; see \cite{DuChMu24} for example. 

\begin{remark}
We can also show that the sample Fr\'echet mean for $\mathbb{E}_\oplus[F]$ based on SVRVs $F_1,\ldots,F_n$ taking values in $K_{kc}^R(\mathbb R^d)$ coincides with the conventional sample Minkowski mean, which is a sample counterpart of the Aumann expectation. More precisely, the sample Fr\'echet mean $\mu_{\oplus,n}$ of $F_1,\dots, F_n$ with respect to the metric $d_{kc}$ is defined as a minimizer of the sample Fr\'echet function $\mathcal{Q}_n(\nu)={1 \over n}\sum_{i=1}^nd_{kc}^2(\nu,F_i)$, that is
$\mu_{\oplus,n} \in \mathrm{argmin}_{\nu \in K_{kc}^R(\mathbb{R}^d)}\mathcal{Q}_n(\nu)$.     
\end{remark}

\begin{remark}
The convexity restriction is essential for Proposition~\ref{prp: SVRV-exp-cexp}. For example, consider the deterministic compact set $A=\{-1,1\}\subset\mathbb R$. Then $A$ and its convex hull $\operatorname{co}(A)=[-1,1]$ have the same support function $s(p,A)=s(p,\operatorname{co}(A))$ for $p\in\mathbb S^0$. Hence, $d_{kc}(A,\operatorname{co}(A))=0$ although $A\ne\operatorname{co}(A)$. Thus, on the class of general compact sets, $d_{kc}$ is no longer a metric, since the support function identifies a set only through its closed convex hull. Consequently, the equivalence between the Aumann expectation and the Fr\'echet mean established in Proposition~\ref{prp: SVRV-exp-cexp} cannot in general be extended to nonconvex compact sets.
\end{remark}

The equivalence of the sample Minkowski and Fr\'echet means is presented as follows.
\begin{corollary}\label{cor:sample-F-A-exp}
Let $F_1,\dots, F_n$ be SVRVs taking values in $K_{kc}^R(\mathbb{R}^d)$. Then there exists a unique sample Fr\'echet mean $\mu_{\oplus,n}$ with respect to $d_{kc}$ and it coincides with the sample Minkowski mean, that is, $\mu_{\oplus,n}={1 \over n} \bigoplus_{i=1}^n F_i$, where $\bigoplus_{i=1}^n F_i$ is the Minkowski sum of $F_1,\dots,F_n$. 
\end{corollary}

\section{Metric statistics for random sets}\label{sec:stat}
This section presents several methods developed within metric statistics that can be applied to the analysis of SVRVs. For $t \in \mathbb{R}$ and a set $A$, define $t A = \{ta: a \in A\}$. Let $\mathcal{X} \subset \mathbb{R}^k$ be a compact set.

\subsection{Linear regression as a weighted Fr\'echet mean}

To motivate, consider the linear projection model for an Euclidean outcome $Y$ taking values in $\mathbb{R}$ and Euclidean covariates $X$ taking values in $\mathbb{R}^k$,  
\begin{equation}
    Y=\theta_0^*+(\theta_1^*)'(X-\mu)+\varepsilon,\quad\mathbb{E}[X\varepsilon]=0,
\end{equation}
where $\mu=\mathbb{E}[X]$ and the vector $(\theta_0^*,\theta_1^*)$ solves
\begin{equation}
(\theta_0^*,\theta_1^*)=\mathrm{argmin}_{(\theta_0,\theta_1) \in\mathbb{R}^{k+1}}\mathbb{E}[(\mathbb{E}[Y|X]-\{\theta_{0}+\theta_{1}'(X-\mu)\})^{2}].
\end{equation}
Let $\sigma_{YX}=\mathbb{E}[Y(X-\mu)]$ and $\Sigma=\mathbb{E}[(X-\mu)(X-\mu)']$. Assume that $\Sigma$ is invertible. Then we have $\theta_0^*=\mathbb{E}[Y]$ and $\theta_1^*=\Sigma^{-1}\sigma_{YX}$, and obtain the regression function
\begin{align}
m(x) & =\theta_0^*+(\theta_1^*)'(x-\mu)
 =\mathbb{E}[Y]+(x-\mu)\Sigma^{-1}\sigma_{YX}\nonumber \\
& =\mathbb{E}[Y+(x-\mu)'\Sigma^{-1}(X-\mu)Y]
 =\mathbb{E}[w(x,X)Y],
\end{align}
where $w(x,z)=1+(x-\mu)'\Sigma^{-1}(z-\mu)$. Since $\mathbb{E}[w(x,X)]=1$, the regression function $m(x)$ is characterized as the solution 
\begin{equation}
m(x)=\mathrm{argmin}_{y\in\mathbb{R}}\mathbb{E}[w(x,X)d_{E}^2(y,Y)],
\end{equation}
where $d_E$ is the standard Euclidean metric.

By extending the above representation of $m(x)$, the global Fr\'echet regression (GFR) function for an outcome $Y$ in a general metric space $(\mathcal{M},d)$ is defined as follows \citep{PeMu19}:
\begin{equation}
m_{\oplus}(x)=\mathrm{argmin}_{\nu\in\mathcal{M}}\mathbb{E}[w(x,X)d^{2}(\nu,Y)].
\end{equation}
Let $\{Y_i,X_i\}_{i=1}^n$ be an i.i.d. sample from the joint distribution of $(Y,X)$ taking values in $\mathcal{M}\times \mathbb{R}^k$. Then the GFR estimator is given as the sample counterpart of $m_{\oplus}$, that is
\begin{align}\label{eq:sGFR-def}
\hat{m}_{\oplus}(x) & =\mathrm{argmin}_{\nu\in\mathcal{M}}{1\over n}\sum_{i=1}^n\hat{w}(x,X_i)d^{2}(\nu,Y_i),
\end{align}
where $\hat{w}(x,z)=1+(x-\bar{X})'\hat{\Sigma}^{-1}(z-\bar{X})$, $\bar{X}=n^{-1}\sum_{i=1}^n X_i$, and $\hat{\Sigma}=n^{-1}\sum_{i=1}^n (X_i-\bar{X})(X_i-\bar{X})'$.

\subsection{Global Fr\'echet regression for random sets}

We now apply the GFR for a pair of random elements $(F,X) \in K_{kc}^R(\mathbb{R}^d) \times \mathcal{X}$, which is defined as 
\begin{equation}
m_{\oplus}(x)=\mathrm{argmin}_{\nu\in K_{kc}^R(\mathbb{R}^d)}\mathbb{E}[w(x,X)d_{kc}^2(\nu,F)].
\end{equation}
In order to derive the explicit form of $m_{\oplus}(x)$, recall $\Psi: K_{kc}^R(\mathbb{R}^d) \ni F \mapsto s(\cdot, F) \in L^2(\mathbb{S}^{d-1})$, and define 
\begin{equation}
    m_{\oplus,\Psi}(x) = \mathrm{argmin}_{h \in L^2(\mathbb{S}^{d-1})}\mathbb{E}[w(x,X)\|h - \Psi(F)\|^{2}_{2,\mathbb{S}^{d-1}}].
\end{equation}
Then we have 
\begin{equation}
m_{\oplus,\Psi}(x)
= \mathbb{E}[w(x,X)\Psi(F)]
= \mathbb{E}[w(x,X)\Psi(\mathbb{E}_\oplus[F|X])]
= \mathbb{E}[w(x,X)s(\cdot, \mathbb{E}_\oplus[F|X])],
\end{equation}
where the second equality follows from Proposition \ref{prp:F-mean-embed}. 

Assume that $\inf_{z \in \mathcal{X}}w(x,z) \geq 0$. Then from  Proposition \ref{prp: SVRV-exp-cexp} (ii) and Lemma \ref{lem:sf-add-conv} (i), we have
\begin{equation}
m_{\oplus,\Psi}(x) 
= \mathbb{E}[s(\cdot, w(x,X) \mathbb{E}_\oplus[F|X])]
= s(\cdot, \mathbb{E}[w(x,X)  F])
= \Psi(\mathbb{E}[w(x,X) F]),
\end{equation}
which yields the expression of $m_\oplus(x)$ by the Aumann expectations:
\begin{equation}
    m_\oplus(x) = \mathbb{E}[w(x,X) F].
\end{equation}
When $d=1$, so that $F$ is interval-valued as in \citet[Section 4]{BeMo08}, 
\begin{align}
m_\oplus(x) 
&= \mathbb{E}[w(x,X) F] \notag\\
&= (1, (x-\mu)')\left(
\begin{array}{cc}
1 & 0 \\
0 & \Sigma \\
\end{array}\right)^{-1} 
\left\{\begin{pmatrix}
    \mathbb E[f]\\ \mathbb E [(X-\mu) f]
\end{pmatrix} : f \in \mathcal S^1(F)\right\}
=  (1, (x-\mu)')\Theta, \label{eq:bemo}
\end{align}
where the second and third equalities hold regardless of the sign of $w(x,X)$. 
Here, $\Theta$ corresponds to the population set valued best linear predictor in \cite{BeMo08}.
Hence, when $\inf_{z \in \mathcal{X}}w(x,z) \geq 0$, the GFR prediction coincides with the set-valued best linear prediction of \citet{BeMo08}.

On the other hand, $\inf_{z\in \mathcal X} w(x,z)$ may take negative values as $x$ moves farther from $\mu$ in a given direction. For example, when $X$ is scalar, for $x>\mu$, $\inf_{z\in\mathcal X} w(x,z) = 1+(x-\mu)(\inf\mathcal{X}-\mu)/\mathbb{V}[X]$, which decreases as $x$ moves farther above $\mu$ and may eventually become negative.
If $\inf_{z \in \mathcal{X}}w(x,z) < 0$, we have, pointwise on \(\mathbb S^{d-1}\),
\begin{align}\label{ineq:GFR-BLP}
m_{\oplus,\Psi}(x) \leq \mathbb{E}[s(\cdot, w(x,X)\mathbb{E}_\oplus[F|X])] =  \Psi(\mathbb{E}[w(x,X)F]).
\end{align}
Specifically, it is possible to provide an example such that $m_{\oplus,\Psi}(x)< \Psi(\mathbb{E}[w(x,X)F])$ (see Appendix \ref{subsec:app-counter-exm}).  
The difference arises from the treatment of negative weights. The conventional predictor $\mathbb{E}[w(x,X)F]$ applies the weight directly to the set $F$. For any scalar weight $w$, $wF = \{wf : f \in F\}$ remains a convex set, and $\mathbb{E}[w(x,X)F]$ averages such valid convex sets. 
In contrast, $m_{\oplus,\Psi}(x)=\mathbb{E}[w(x,X)s(\cdot,\mathbb{E}_\oplus[F|X])]$ applies the weight to the support function. When $w(x,X)<0$, $w(x,X)s(\cdot,\mathbb{E}_\oplus[F|X])$ need not itself be a valid support function, so $m_{\oplus,\Psi}(x)$ averages such signed functions in $L^2(\mathbb{S}^{d-1})$. Consequently, $m_{\oplus,\Psi}(x)$ need not coincide with $\Psi(\mathbb{E}[w(x,X)F])$ and may even fail to be a valid support function.

Therefore, for the GFR with a general weight function $w(x,z)$, we compute $m_\oplus(x)$ by projecting $m_{\oplus,\Psi}(x)$ on $\Psi(K_{kc}^R(\mathbb{R}^d))$, that is, 
\begin{equation} \label{eq:mop}
m_\oplus(x) = \Psi^{-1}(\widetilde{m}_{\oplus,\Psi}(x)),\quad \widetilde{m}_{\oplus,\Psi}(x)= \mathrm{argmin}_{h \in \Psi(K_{kc}^R(\mathbb{R}^d))}\|h - m_{\oplus,\Psi}(x)\|_{2,\mathbb{S}^{d-1}}. 
\end{equation}
The following result shows that $m_\oplus(x)$ is in fact the GFR function.  
\begin{proposition}\label{prp:F-Reg-solution}
Let $\widetilde m_{\oplus,\Psi}(x)$ be defined as in (\ref{eq:mop}). $m_\oplus(x)=\Psi^{-1}(\widetilde{m}_{\oplus,\Psi}(x))$ is the GFR function, that is, $\Psi^{-1}(\widetilde{m}_{\oplus,\Psi}(x))$ is the unique minimizer of $\mathbb{E}[w(x,X)d_{kc}^2(\nu,F)]$ over $\nu \in K_{kc}^R(\mathbb{R}^d)$. Additionally, if $\sup_{z \in \mathcal{X}}|w(x,z)|<\infty$ and $m_{\oplus,\Psi}(x) \in \Psi(K_{kc}(\mathbb{R}^d))$, then $m_\oplus(x) \subset \mathbb{E}[w(x,X)F]$. 
\end{proposition}

Figure \ref{fig:proj} illustrates the relation between $m_\oplus(x)$ and $\mathbb{E}[w(x,X)F]$ when $d=1$. In the one-dimensional case, $F$ can be written as $F=[-a,b]$ with $a+b\ge0$. The support function $s(p,F)$ of $F\in K_{kc}(\mathbb{R})$ is completely characterized by $s(1,F)$ and $s(-1,F)$ and satisfies $b=s(1,F)\ge -s(-1,F)=-a$. Hence, the image $\Psi(K_{kc}(\mathbb{R})) = \{s(\cdot,F):F\in K_{kc}(\mathbb{R})\}$, namely, the set of support functions of all compact convex subsets of $\mathbb{R}$, can be identified with $\{(a,b)\in\mathbb{R}^2:a+b\ge0\}$. The blue region $R_{+,+}$, the purple region $R_{-,+}$, and the red region $R_{-,-}$ correspond to the sets $K_1$, $K_2$, and $K_3$, respectively, where $K_1\cup K_2\cup K_3=K_{kc}(\mathbb{R})$, $K_1=\{[x,y]:0<x\le y<\infty\}$, $K_2=\{[x,y]:-\infty<x\le0\le y<\infty\}$, and $K_3=\{[x,y]:-\infty<x\le y<0\}$. 
The filled square denotes $\Psi(\mathbb{E}[w(x,X)F])$, the filled circle denotes $m_{\oplus,\Psi}(x)$, and the open circle denotes $\widetilde m_{\oplus,\Psi}(x)$.
Figure~\ref{fig:a} represents the case $m_{\oplus,\Psi}(x)\in\Psi(K_{kc}(\mathbb{R}))$, where no projection is required.
Figures~\ref{fig:b} and \ref{fig:c} represent the cases $m_{\oplus,\Psi}(x)\notin\Psi(K_{kc}(\mathbb{R}))$, where Proposition~\ref{prp:F-Reg-solution} shows that the GFR is obtained by projecting the unconstrained Hilbert-space target $m_{\oplus,\Psi}(x)$ onto $\Psi(K_{kc}(\mathbb{R}))$. 
Figure~\ref{fig:b} corresponds to the case $m_{\oplus}(x)\subset\mathbb{E}[w(x,X)F]$, whereas Figure~\ref{fig:c} corresponds to the case $m_{\oplus}(x)\not\subset\mathbb{E}[w(x,X)F]$. 
\begin{figure}
    \centering

    \begin{subfigure}[t]{0.32\textwidth}
        \centering
        \includegraphics[width=\linewidth]{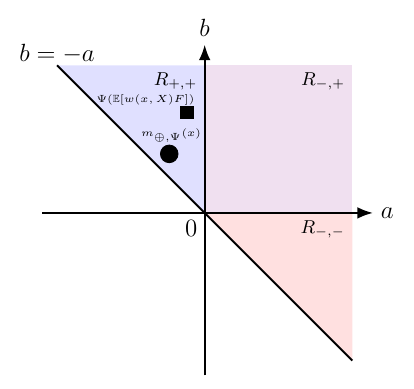}
        \caption{$m_{\oplus,\Psi}(x)\in\Psi(K_{kc}(\mathbb{R}))$}
        \label{fig:a}
    \end{subfigure}
    \hfill
    \begin{subfigure}[t]{0.32\textwidth}
        \centering
        \includegraphics[width=\linewidth]{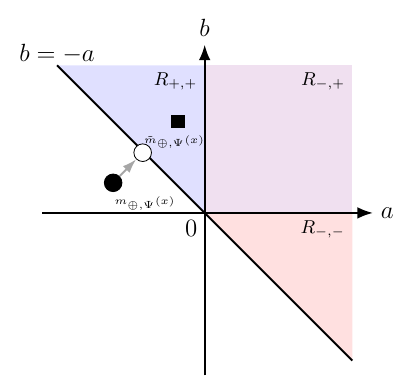}
        \caption{$m_{\oplus,\Psi}(x)\notin\Psi(K_{kc}(\mathbb{R}))$, $m_{\oplus}(x)\subset\mathbb{E}[w(x,X)F]$}
        \label{fig:b}
    \end{subfigure}
    \hfill
    \begin{subfigure}[t]{0.32\textwidth}
        \centering
        \includegraphics[width=\linewidth]{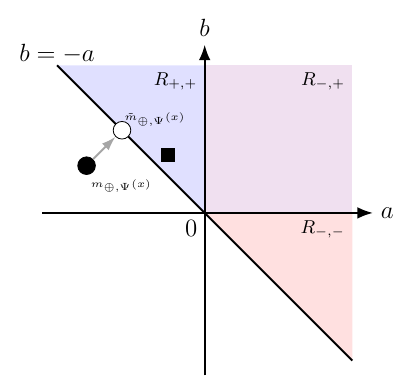}
        \caption{$m_{\oplus,\Psi}(x)\notin\Psi(K_{kc}(\mathbb{R}))$, $m_{\oplus}(x)\not\subset\mathbb{E}[w(x,X)F]$}
        \label{fig:c}
    \end{subfigure}

    \caption{Illustration of the relation between $m_\oplus(x)$ and $\mathbb{E}[w(x,X)F]$ when $d=1$.}
    \label{fig:proj}
    \begin{flushleft}
    \footnotesize
    \renewcommand{\baselineskip}{11pt}
    \textbf{Note:} The shaded region represents the class of support functions. The filled square denotes $\Psi(\mathbb{E}[w(x,X)F])$, the filled circle denotes $m_{\oplus,\Psi}(x)$, and the open circle denotes $\widetilde m_{\oplus,\Psi}(x)$. Panel (a) considers the case where $m_{\oplus,\Psi}(x)$ is a valid support function, while panels (b) and (c) consider the cases where it is not, with and without the subset relation, respectively.
    \end{flushleft}
\end{figure}

In practice, the projection step in \eqref{eq:mop} can be implemented numerically using, for example, the algorithm proposed by \citet[Section 4]{Gardner09}.

Let $(F,X) \in K_{kc}^R(\mathbb{R}^d) \times \mathcal{X}$ be a pair of random elements and $\{F_i,X_i\}_{i=1}^n$ be an i.i.d. sample from the joint distribution of $(F,X)$. The uniform convergence rate of the GFR estimator in (\ref{eq:sGFR-def}) can be established as follows.

\begin{proposition}\label{prp:GFR-rate}
Assume that $d \leq 4$. For a given $B_0>0$, we have 
\begin{equation}
\sup_{\|x\|\leq B_0}d_{kc}(\hat{m}_\oplus(x),m_\oplus(x)) = 
\begin{cases}
O_p(n^{-{1 \over 2(\alpha-1)}}) & \text{for any $\alpha>2$ when $d=1$},\\
O_p(n^{-{1 \over 2(1+(d-1)/4)}}) & \text{when $d\in \{2,3,4\}$}.\\
\end{cases}
\end{equation}
\end{proposition}

In Proposition \ref{prp:GFR-rate}, since $\alpha>2$ can be chosen arbitrarily, the convergence rate of the GFR estimator can be arbitrarily close to the parametric rate when $d=1$. The restriction $d\leq 4$ is due to the complexity of the metric space $(K_{kc}(\mathbb{R}^d),d_{kc})$. Specifically, let $N(\varepsilon,\mathcal{F},d)$ denote the covering number of $\mathcal{F}$ with radius $\varepsilon$ under the metric $d$. The required covering number bounds are summarized in Lemma~\ref{lem:covering-number}. In particular, for $d\ge2$, the entropy integral condition in Proposition~\ref{prp:GFR-rate} reduces to $\int_0^1 \sqrt{\varepsilon^{-(d-1)/2}}\,d\varepsilon<\infty$, which holds if and only if $d\le4$. See the proof of Proposition~\ref{prp:GFR-rate} for details. The same comment applies to Proposition~\ref{prp:IPW-conv}.

\subsubsection{Illustration}
As an illustration, consider the case where the SVRV $F$ takes values in $K_{kc}^R(\mathbb{R})$. In this case, letting $F=[L,U]$, we have
\begin{equation}
    m_{\oplus,\Psi}(x) =\mathbb{E}[w(x,X) s(p, F)]\text{ for }p\in\{-1,1\} = \begin{cases}
        -\mathbb{E}[w(x,X) L] & \text{ for } p=-1\\
        \mathbb{E}[w(x,X) U] & \text{ for } p=1
    \end{cases}.
\end{equation}
Also since $w(x,X)F=[\min\{w(x,X)L, w(x,X)U\}, \max\{w(x,X)L, w(x,X)U\}]$, we obtain
\begin{equation}
     \Psi(\mathbb{E}[{w(x,X)F}]) = \begin{cases}
        -\mathbb{E}[\min\{w(x,X)L, w(x,X)U\}] & \text{ if } p=-1 \\
        \mathbb{E}[\max\{w(x,X)L, w(x,X)U\}] & \text{ if } p=1\\
    \end{cases}.
\end{equation}
Following \cite{BeMo08}, we consider the best linear prediction of the returns to education of log-wages for white men between the ages of 20 and 50.
We artificially divide the support of the outcome variable into five equal intervals to construct interval data. 
We employ the March 2009 Current Population Survey data from \cite{Hansen:2022_econ}, and treat them as the population. Then we compare the GFR $m_{\oplus}(x)= \Psi^{-1}(\widetilde{m}_{\oplus,\Psi}(x))$ and the Aumann expectation $\mathbb{E}[{w(x,X)F}]$ at $x=\mu\approx 13.727$ and $x=16$, corresponding to the Bachelor's degree. 

\begin{figure}[t]
    \centering
    \includegraphics[width=0.6\linewidth]{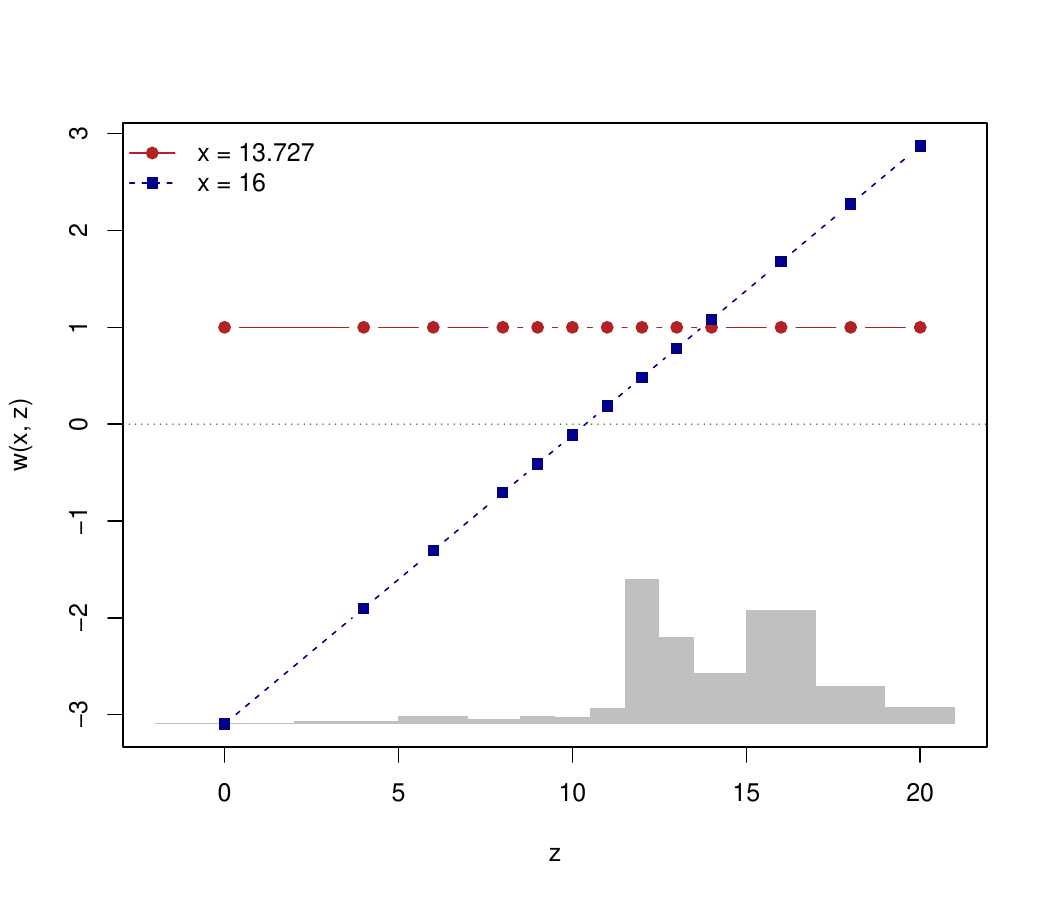}
    \caption{Weight functions $w(\mu,z)$ and $w(16,z)$} and histogram of $X$.
    \label{fig: cps}
\end{figure}
For $x=\mu$, since $w(\mu,z)=1+(\mu-\mu)'\Sigma^{-1}(z-\mu)=1$ (see Figure \ref{fig: cps}), we have $m_{\oplus}(x)=\mathbb{E}[{w(x,X)F}]$.
In this case, the Aumann expectation and the GFR are computed as
\begin{equation}
    \mathbb{E}[{w(\mu,X)F}]= [9.624, 12.253]\quad \text{and}\quad m_{\oplus}(\mu) = [9.624, 12.253].
\end{equation}

On the other hand, Figure \ref{fig: cps} shows that $\inf_{z \in \mathcal{X}}w(16,z) < 0$. 
We can also confirm that $m_{\oplus,\Psi}(16)$ is a proper support function, i.e., $m_{\oplus,\Psi}(16) \in \Psi(K_{kc}(\mathbb{R}))$.
These suggest that $m_{\oplus}(x)\subset\mathbb{E}[{w(x,X)F}]$ (Proposition \ref{prp:F-Reg-solution}).
Indeed, we obtain that
\begin{equation}
    \mathbb{E}[{w(16,X)F}] = [9.891, 12.778] \quad\text{and}\quad m_{\oplus}(16) = [10.020, 12.649].
\end{equation}

\subsection{Errors-in-variables}

The GFR approach can be naturally extended to errors-in-variables models. Suppose the object of interest is 
\begin{equation}
m_\oplus(x)=\mathrm{argmin}_{\nu\in K_{kc}^R(\mathbb{R}^d)}\mathbb{E}[w(x,X)d_{kc}^2(\nu,F)].
\end{equation}
However, Euclidean covariates $X$ are mismeasured, and we instead observe $W = X + \varepsilon$, where $\varepsilon$ is a measurement error such that $\mathbb{E}[\varepsilon|X,F]=0$ and $\mathbb{E}[\varepsilon \varepsilon'] = \Sigma_\varepsilon$. In this setup, $\mathbb{E}[W]=\mu$ and $\Sigma_W = \mathbb{E}[(W - \mathbb{E}[W])(W - \mathbb{E}[W])']=\Sigma + \Sigma_\varepsilon$. Therefore, 
\begin{equation}
\mathbb{E}[\{1 + (x-\mathbb{E}[W])'\Sigma_W^{-1}(W - \mathbb{E}[W])\}d_{kc}^2(\nu,F)]
= \mathbb{E}[\{1 + (x-\mu)'(\Sigma + \Sigma_\varepsilon)^{-1}(X - \mu)\}d_{kc}^2(\nu,F)].
\end{equation}
Then the GFR estimator (\ref{eq:sGFR-def}) with $\mathcal{M}=K_{kc}^R(\mathbb{R}^d)$, $d=d_{kc}$, and $Y_i=F_i$ is not consistent for $m_\oplus(x)$ under $d_{kc}$.

Suppose we observe instrumental variables $Z = X + v$ (or repeated measurements on $X$) such that $\mathbb{E}[v|X,F]=0$ and $\mathbb{E}[v\varepsilon'|X,F]=0$.  Then we can recover $\Sigma$ as 
\begin{equation}
    \Sigma_{ZW}=\mathbb{E}[(Z-\mathbb{E}[Z])(W-\mathbb{E}[W])'] = \mathbb{E}[(X-\mu + v)(X - \mu + \varepsilon)'] = \Sigma.
\end{equation}
Therefore, the GFR function $m_\oplus(x)$ can be estimated based on the alternative representation
\begin{equation}
    m_\oplus(x)=\mathrm{argmin}_{\nu\in K_{kc}^R(\mathbb{R}^d)}\mathbb{E}[\tilde{w}(x,Z)d_{kc}^2(\nu,F)],
\end{equation}
where $\tilde{w}(x,z) = 1 + (x-\mu)'\Sigma_{ZW}^{-1}(z -\mu)$.

When the corrected weight $\tilde w(x,Z)$ is nonnegative, the above errors-in-variables correction can also be formulated in terms of Aumann expectations, that is, $m_\oplus(x) = \mathbb{E}[\tilde w(x,Z) F]$. 
When the corrected weight $\tilde w(x,Z)$ takes negative values, however, the equality need not hold.
Furthermore, $\mathbb{E}[\tilde w(x,Z) F]$ is in general a superset of $\mathbb E[w(x, X) F]$ since
\begin{align}
     \mathbb E[s(p,\, \tilde w(x, Z)F)]
    & = \mathbb E[\mathbb E[s(p,\, \tilde w(x, Z)F) \mid X, F]]\notag\\
    &\ge \mathbb E[s(p,\, \mathbb E[\tilde w(x, Z)\mid X, F] F)]
    = \mathbb E[s(p,\, w(x, X) F)]
\end{align}
by Jensen's inequality.
Accordingly, the estimable Aumann expectation $\mathbb{E}[\tilde w(x,Z) F]$ need not coincide with either $m_\oplus(x)$ or the Beresteanu--Molinari predictor $(1,(x-\mu))' \Theta$ defined in \eqref{eq:bemo}.

The inclusion $(1,(x-\mu))' \Theta \subseteq\mathbb{E}[\tilde w(x,Z) F]$ also points to a limitation of directly extending the standard Beresteanu--Molinari approach to the errors-in-variables setting. In particular, for $d=1$, $\mathbb{E}[\tilde w(x,Z) F] = (1,(x-\mu)') \tilde\Theta$, where 
\begin{align}
    \tilde \Theta = 
    \begin{pmatrix}
    1 & 0 \\
    0 & \Sigma_{ZW} 
    \end{pmatrix}^{-1}\left\{\begin{pmatrix}
        \mathbb E[f]\\ \mathbb E[(Z-\mu)f] 
    \end{pmatrix} : f\in \mathcal S^1(F)\right\},\label{eq:outer}
\end{align}
which appears to be a natural errors-in-variables analogue of equation (4.2) of \cite{BeMo08}. In contrast, the exact infeasible Beresteanu--Molinari coefficient set is
\begin{align}
    \Theta = \begin{pmatrix}
        1 & 0\\ 0 & \Sigma
    \end{pmatrix}^{-1}\left\{\begin{pmatrix}
        \mathbb E[f]\\ \mathbb E[(X-\mu)f] 
    \end{pmatrix} : f\in \mathcal S^1(F)\right\}.\label{eq:sharp}
\end{align}
For any $f \in \mathcal S^1(F)$, define $g= \mathbb E[f \mid X,F]$.
Then $g\in F$ almost surely, $\mathbb E[g] = \mathbb E[f]$, and $\mathbb E[(X-\mu)f] = \mathbb E[(X-\mu)g] = \mathbb E[(Z-\mu)g]$. Then, together with $\Sigma = \Sigma_{ZW}$, $\Theta$ is rewritten as
\begin{align}
    \Theta = \begin{pmatrix}
        1 & 0\\ 0 & \Sigma_{ZW}
    \end{pmatrix}^{-1}\left\{\begin{pmatrix}
        \mathbb E[g]\\ \mathbb E[(Z-\mu)g] 
    \end{pmatrix} : g\in \mathcal S^1_{X,F}(F)\right\},
\end{align}
where $\mathcal S^1_{X,F}(F) = \{g \in \mathcal S^1(F) : g \text{ is } \sigma(X,F)\text{-measurable}\}$. The inclusion $\Theta \subseteq \tilde\Theta$ reflects the fact that $\mathcal S^1_{X,F}(F) \subseteq \mathcal S^1(F)$.
In fact, under the stated assumptions, the original Beresteanu--Molinari coefficient set $\Theta$ is not generally identified from the joint distribution of the observables $(W,Z,F)$, as illustrated by a counterexample in Appendix~\ref{ex:nonsharp}. $\tilde\Theta$ provides an outer bound on the identified set of coefficient vectors, but this bound need not be sharp (see Appendix~\ref{ex:nonsharp}).
Consequently, \eqref{eq:outer} and \eqref{eq:sharp} need not coincide without additional restrictions, such as $\mathbb E[vf]=0, \forall f \in \mathcal S^1(F)$. Hence, the Beresteanu--Molinari approach (\citealp[Section 4]{BeMo08}) may not directly apply to errors-in-variable settings.

Other possible statistical methods for analyzing set-valued response variables with Euclidean covariates using the GFR or related regression models include \cite{TuWuMu23} and \cite{KuOt25} for variable selection and model averaging, \cite{BhMu23} for an extension of the GFR to single index models, \cite{PeMu19} and \cite{ChMu22} for extensions of the GFR to local linear regression. 

\subsection{Missing data}

Consider the missing at random setup. Let $T \in \{0,1\}$ be the indicator for missing data. We observe $TF$ and $X$, where $(F,X) \in K_{kc}^R(\mathbb{R}^d) \times \mathcal{X}$ and assume $F$ and $T$ are independent given $X$. Let $\{F_i,X_i,T_i\}_{i=1}^n$ be an i.i.d. sample from the joint distribution of $(F,X,T)$. We are interested in the Fr\'echet mean of $F$, that is, $\mathbb{E}_\oplus[F] = \mathrm{argmin}_{\nu \in K_{kc}^R(\mathbb{R}^d)}\mathbb{E}[d_{kc}^2(\nu,F)]$.

Letting $e(x) = \mathbb{P}(T=1|X=x)$, we can see that 
\begin{equation}
\mathbb{E}\left[{T \over e(X)}d_{kc}^2(\nu,F)\right]
= \mathbb{E}\left[{1 \over e(X)}\mathbb{E}[Td_{kc}^2(\nu,F)|X]\right]
= \mathbb{E}\left[{1 \over e(X)}\mathbb{E}[T|X]\mathbb{E}[d_{kc}^2(\nu,F)|X]\right]
= \mathbb{E}[d_{kc}^2(\nu,F)]. 
\end{equation}
Therefore, the Fr\'echet mean of $F$ can be alternatively written as 
\begin{equation}
    \mathbb{E}_\oplus[F] = \mathrm{argmin}_{\nu \in K_{kc}^R(\mathbb{R}^d)}\mathbb{E}\left[{T \over e(X)}d_{kc}^2(\nu,F)\right],
\end{equation}
which can be estimated by 
\begin{equation}
\hat{\mathbb{E}}_\oplus[F] = \mathrm{argmin}_{\nu \in K_{kc}^R(\mathbb{R}^d)}{1 \over n}\sum_{i=1}^n {T_i \over \hat{e}(X_i)}d_{kc}^2(\nu,F_i) = \left({1 \over n}\sum_{i=1}^n {T_i \over \hat{e}(X_i)}\right)^{-1}{1 \over n} \bigoplus_{i=1}^n {T_i \over \hat{e}(X_i)} F_i,
\end{equation}
and $\hat{e}(x)$ is a nonnegative nonparametric estimator of $e(x)$. 

Since the inverse-probability weights are nonnegative, the estimator in (33) can also be interpreted as a normalized weighted Aumann expectation. Thus, in this setting, the Fr\'echet and Aumann formulations lead to the same estimator.

One can see that $\hat{\mathbb{E}}_\oplus[F]$ is a modified version of inverse probability weighting estimator proposed in \cite{KuZhOtMu24}: 
\begin{align}\label{eq:IPW}
\hat{\Theta}_1^{(\mathrm{IPW})}={1 \over n} \bigoplus_{i=1}^n {T_i \over \hat{e}(X_i)} F_i. 
\end{align} 
See Appendix \ref{subsec:app-IPW} for the derivation of (\ref{eq:IPW}). The following result provides the convergence rate of $\hat{\mathbb{E}}_\oplus[F]$.
\begin{proposition}\label{prp:IPW-conv}
Suppose that the following conditions hold.  
\begin{itemize}
\item[(i)] There exists a constant $\eta_0 \in (0,1/2)$ such that $\eta_0 \leq e(x) \leq 1-\eta_0$ for all $x \in \mathcal{X}$.
\item[(ii)] $\sup_{x \in \mathcal{X}}|\hat{e}(x)-e(x)|=O_p(\rho_n)$, $\rho_n \to 0$ as $n \to \infty$. 
\item[(iii)] As $\delta \to 0$, $\int_0^1 \sqrt{\log N(\delta \varepsilon, B_{\delta'}(e), \|\cdot\|_\infty)}d\varepsilon = O(\delta^{-\varpi_1})$ for some $\delta'>0$ and $\varpi_1 \in (0,1)$, where $\|e_1-e_2\|_\infty=\sup_{x \in \mathcal{X}}|e_1(x)-e_2(x)|$ for $e_1, e_2: \mathcal{X} \to \mathbb{R}$, and $B_\gamma(e)$ is the ball of radius $\gamma>0$ centered at $e$. 
\end{itemize}
Then for $d\leq 4$, we have $d_{kc}(\hat{\mathbb{E}}_\oplus[F], \mathbb{E}_\oplus[F]) = O_p(n^{-{1 \over 2(1 + \max\{(d-1)/4,\varpi_1\})}} + \rho_n^{\varpi_2})$ for any $\varpi_2 \in (0,1)$. 
\end{proposition} 
If the propensity score $e(x)$ belongs to a class of parametric models $\mathcal{M}_e=\{e(x;\varphi): \varphi \in \Phi\}$ with a compact parameter space $\Phi \subset \mathbb{R}^q$, Condition (iii) is satisfied with $\int_0^1 \sqrt{\log N(\delta \varepsilon, B_{\delta'}(e), \|\cdot\|_\infty)}d\varepsilon = O(-\log\delta)$ as $\delta \to 0$. In this case, it can be shown that convergence rate of $\hat{\mathbb{E}}_\oplus[F]$ is $O_p(n^{-{1 \over 2(\alpha-1)}})$ for any $\alpha >2$ when $d=1$ and $O_p(n^{-{1 \over 2(1 + (d-1)/4)}})$ when $d \in \{2,3,4\}$. For details on this point, see, for example, Section 4.1 of \cite{KuZhOtMu24}.

\subsection{Synthetic control methods}
In causal analysis in economics, synthetic control methods (SCMs) are widely used. Although SCMs are typically formulated for scalar outcomes, the available outcome data may naturally take the form of intervals. For example, \cite{Spencer23} studies the effect of drug decriminalization on overdose deaths using data in which state-month death counts of nine or fewer are suppressed for confidentiality.  He handles these observations by randomly imputing an integer between $0$ and $9$, although they can naturally be represented by the interval $[0,9]$.

The key step in synthetic control analysis is to construct weights such that the weighted ``barycenter'' of the donor units approximates the target quantity as ``closely'' as possible. That is, SCMs require two choices: how to aggregate donor outcomes and how to measure pre-treatment fit. The metric $d_{kc}$ provides a natural and internally coherent solution to both.
First, under $d_{kc}$, the Fr\'echet barycenter of the donor sets coincides with their weighted Aumann expectation by Corollary \ref{cor:sample-F-A-exp}. Hence, the $d_{kc}$-based barycenter retains the standard aggregation interpretation from random set theory. Second, pre-treatment fit is assessed using the same metric that induces this barycenter. Thus, both donor aggregation and the criterion used to select the synthetic-control weights are governed by a common geometric structure.

Let $F_{i,t} \in K_{kc}^R(\mathbb{R}^d)$ for $i\in\{0,\ldots,K\}$ and $t\in\{1,\ldots,T,T+1\}$, where $i$ indexes units and $t$ indexes time periods.
We assume that $t=1,\ldots,T$ are pre-treatment periods and that only unit $i=0$ is treated in period $t=T+1$.
Our goal is to recover the counterfactual set-valued outcome of unit $i=0$ in period $t=T+1$ had it not received the treatment.
Following \cite{KuZhOtMu25a}, who extend conventional synthetic control methods to random objects, we construct the synthetic control as follows:
\begin{align}
    F_{0,T+1}^{(\mathtt{SC}, d)} = \argmin_{\nu \in K_{kc}^R(\mathbb{R}^d)} \sum_{i=1}^K \bar{w}_i^{(d)} d^2\left(\nu, F_{i,T+1}\right),
\end{align}
\begin{align}
    \bar{\bm{w}}^{(d)} \in \argmin_{\bm{w} \in \Delta^{K-1}} \frac{1}{T}\sum_{t=1}^{T}d^2\left(F_{0,t}, F_{1:K,t}^{(\bm w, d)}\right),\quad
    F_{1:K,t}^{(\bm w, d)} = \argmin_{\nu \in K_{kc}^R(\mathbb{R}^d)} \sum_{i=1}^K w_i d^2\left(\nu, F_{i,t}\right),
\end{align}
where $\Delta^{K-1}$ denotes the unit simplex.
If we employ the metric $d_{kc}$, inspection of the proof of Corollary \ref{cor:sample-F-A-exp} shows that its result extends immediately to establish the equivalence between the nonnegatively weighted sample Fr\'echet mean and the corresponding weighted sample Aumann expectation, or, $\mathrm{arg\,min}_{\nu \in K_{kc}^R(\mathbb{R}^d)} \sum_{i=1}^K w_i d_{kc}^2(\nu, F_{i}) = \bigoplus_{i=1}^K w_i F_{i}$.
Therefore, under the metric $d_{kc}$, the synthetic control for SVRVs is given by
\begin{align}
    F_{0,T+1}^{(\mathtt{SC}, d_{kc})} = \bigoplus_{i=1}^K \bar w_i^{(d_{kc})} F_{i,T+1},\quad
    \bar{\bm{w}}^{(d_{kc})} \in \argmin_{\bm{w} \in \Delta^{K-1}} \frac{1}{T}\sum_{t=1}^{T}d_{kc}^2\left(F_{0,t}, \bigoplus_{i=1}^K w_i F_{i,t}\right).\label{eq:scm}
\end{align}
That is, we predict the counterfactual set in the post-treatment period using a weighted Aumann expectation of the donor outcomes, where the weights are chosen so that the corresponding weighted Aumann expectation closely approximates the target set over the pre-treatment periods.
This provides a natural extension of conventional synthetic control methods for scalar outcomes to set-valued outcomes, using the Aumann expectation as a summary of donor SVRVs.

The use of the metric $d_{kc}$ in \eqref{eq:scm} allows us to interpret the fitted pre-treatment trajectory $\{\bigoplus_{i=1}^K \bar{w}_i^{(d_{kc})} F_{i,t}\}_{t=1}^T$ as the metric projection of $\{F_{0,t}\}_{t=1}^T$ onto the convex hull of the donor trajectories in a product Hilbert space.
This interpretation is attractive from a practical perspective as well. Since projections onto closed convex subsets of a Hilbert space are unique, the fitted pre-treatment trajectory is uniquely determined. This uniqueness need not hold if a different metric is used to assess pre-treatment fit.
For example, consider a single pre-treatment period with target set $[0,2]$ and donor sets $[1,1]$ and $[1,3]$. Their Aumann convex combinations are $(1-\lambda)[1,1]\oplus \lambda[1,3]=[1,1+2\lambda]$ for $\lambda\in[0,1]$. Under the Hausdorff metric, $d_H([0,2],[1,1+2\lambda])=\max\{1,|1-2\lambda|\}=1$ for every $\lambda\in[0,1]$. Hence, every convex combination is a best fit, so the fitted synthetic outcome is not uniquely determined.

\subsection{Other applications}
The previous subsections illustrate how methods from metric statistics can be applied to SVRVs in the conventional regression setting with Euclidean predictors and set-valued responses, including extensions to the errors-in-variables and missing-data models. We have also shown how the metric-statistical framework leads to a natural extension of synthetic control methods to set-valued outcomes.

These examples represent only a small subset of the potential applications. An important advantage of establishing a metric-statistical foundation for the Aumann expectation is that methods developed for general random objects can be imported into the econometric analysis of SVRVs.

Several problems of particular importance in econometrics remain largely unexplored for SVRVs. Examples include semiparametric single-index models, correction for endogenous sample selection, and causal inference methods such as regression discontinuity and difference-in-differences designs. By contrast, corresponding methods are already available in the metric-statistics literature for general random objects (see, e.g., \citealp{BhMu23, kuri:26, ZhKuOtMu25, KuZhOtMu25b}). Rather than developing separate SVRV-specific procedures from scratch, these methods can therefore be imported directly into the random-set setting. Proposition~\ref{prp: SVRV-exp-cexp} ensures that the same metric-statistical machinery carries over to SVRVs.

\newpage
\appendix

\section{Proofs}\label{app:A}

The proofs in this appendix rely on the auxiliary results collected in Appendix \ref{app:B}. An SVRV $F$ is called integrably bounded if $\sup\{\|a\|: a\in F\}$ has finite expectation.

\subsection{Proof of Lemma \ref{lem:d}}

\subsubsection{Proof of (i)} 

For any $A_1,A_2,A_3 \in K_{kc}(\mathbb{R}^d)$, it suffices to verify the following conditions. (a) $d_{kc}(A_1, A_2) = 0 \Leftrightarrow A_1 = A_2$, (b) $d_{kc}(A_1, A_2) = d_{kc}(A_2, A_1)$, and (c) $d_{kc}(A_1, A_2) \leq d_{kc}(A_1,A_3) + d_{kc}(A_3, A_2)$. Conditions (b) and (c) follow immediately from the definition of $d_{kc}$. Now we verify Condition (a). The implication $A_1=A_2 \Rightarrow d_{kc}(A_1,A_2)=0$ is immediate. Conversely, $d_{kc}(A_1,A_2)=0$ if and only if $s(\cdot,A_1)=s(\cdot,A_2)$ almost everywhere on $\mathbb S^{d-1}$. Since $s(\cdot,A_1)$ and $s(\cdot,A_2)$ are continuous by Lemma~\ref{lem:sf-conti}, they coincide everywhere on $\mathbb S^{d-1}$. Hence, Lemma~\ref{lem:sf-cc-unique} yields $A_1=A_2$.

\subsubsection{Proof of (ii)}

First, we verify $\lim_{n \to \infty}d_H(A_n, A_0) = 0 \Rightarrow \lim_{n \to \infty}d_{kc}(A_n, A_0) = 0$. For any bounded closed convex sets, $A,B$, we have $d_{kc}(A, B) \leq \sqrt{\mathrm{vol}(\mathbb{S}^{d-1})}\sup_{p \in \mathbb{S}^{d-1}}|s(p,A) - s(p,B)|$. Then Lemmas \ref{lem:Hd-sf-conv}  and \ref{lem:bcc-set-conv} yield $\lim_{n \to \infty}d_H(A_n, A_0) = 0 \Rightarrow \lim_{n \to \infty}d_{kc}(A_n, A_0) = 0$. 

Next, we verify $\lim_{n \to \infty}d_{kc}(A_n,A_0) = 0 \Rightarrow \lim_{n \to \infty}d_H(A_n,A_0) = 0$. $\sup_{p \in \mathbb{S}^{d-1}}|s(p,A_n)| \leq \sup_n(\sup_{a \in A_n}\|a\|)\leq R$ and hence the sequence $\{s(\cdot, A_n)\}_{n \geq 1}$ is uniformly bounded. Moreover, from the proof of Lemma \ref{lem:sf-conti}, we have $\sup_n|s(u,A_n) - s(v,A_n)| \leq R\|u-v\|$ for any $u,v \in \mathbb{S}^{d-1}$, and hence the sequence $\{s(\cdot, A_n)\}_{n \geq 1}$ is equicontinuous. Then from Arzel\'a-Ascoli theorem, there exist a subsequence $\{s(\cdot, A_{n_k})\}_{k \geq 1}$ and continuous function $t(\cdot)$ on $\mathbb{S}^{d-1}$ such that $\lim_{k \to \infty}\sup_{p \in \mathbb{S}^{d-1}}|s(p,A_{n_k}) - t(p)| = 0$. This yields $\lim_{k \to \infty}\|s(\cdot,A_{n_k}) -t(\cdot)\|_{2,\mathbb{S}^{d-1}} \leq \sqrt{\mathrm{vol}(\mathbb{S}^{d-1})}\lim_{k \to \infty}\sup_{p \in \mathbb{S}^{d-1}}|s(p,A_{n_k}) - t(p)| = 0$. Then we have $\|s(\cdot,A_0) -t(\cdot)\|_{2,\mathbb{S}^{d-1}}
\leq \lim_{k \to \infty}\|s(\cdot,A_0) -s(\cdot, A_{n_k})\|_{2,\mathbb{S}^{d-1}} + \lim_{k \to \infty}\|s(\cdot,A_{n_k}) -t(\cdot)\|_{2,\mathbb{S}^{d-1}} = 0$. This yields $s(p, A_0) = t(p)$ for all $p \in \mathbb{S}^{d-1}$.  Indeed, applying the same argument in the proof of Lemma \ref{lem:d} (i), we have $s(\cdot, A_0) = t(\cdot)$. Then we have $\lim_{n \to \infty}\sup_{p \in \mathbb{S}^{d-1}}|s(p,A_n) - s(p,A_0)| = 0$. Then from Lemma \ref{lem:Hd-sf-conv}, $\lim_{n \to \infty}d_{kc}(A_n,A_0) = 0 \Rightarrow \lim_{n \to \infty}d_H(A_n,A_0) = 0$ and Lemma \ref{lem:bcc-set-conv} yields that $A_0$ is a bounded closed convex set. 

\subsubsection{Proof of (iii)}

The conclusion follows from Lemma \ref{lem:d} (ii) and the fact that $K_{kc}(\mathbb{R}^d)$ is closed with respect to $d_H$ (Lemma \ref{lem:bcc-set-conv}). 

\subsubsection{Proof of (iv)}

It is easy to verify the convexity. Now we show the closedness. 

\noindent
\textbf{(Step 1)} For a sequence of sets $\{A_n\} \subset K_{kc}^R(\mathbb{R}^d)$, assume that $\lim_{n \to \infty}\|s(\cdot,A_n) -t(\cdot)\|_{2,\mathbb{S}^{d-1}} = 0$. From the proof of Lemma \ref{lem:d} (ii), we have that $t(\cdot)$ is a continuous function with $\sup_{p \in \mathbb{S}^{d-1}}|t(p)|\leq \sup_n(\sup_{a \in A_n}\|a\|)\leq R$ and $\lim_{k \to \infty}\sup_{p \in \mathbb{S}^{d-1}}|s(p,A_{n_k}) - t(p)| = 0$ for some subsequence $\{s(\cdot, A_{n_k})\}_{k \geq 1}$. 

\noindent
\textbf{(Step 2)} Now we verify that $t(\cdot)$ is a support function of a bounded closed convex set $A \subset \mathbb{R}^d$. For this, we first extend $s(\cdot, A_{n_k})$ on $\mathbb{S}^{d-1}$ to a function on $\mathbb{R}^d$. Define 
\[
\tilde{s}(x, A_{n_k}) = 
\begin{cases}
\|x\|s\left({x \over \|x\|}, A_{n_k}\right) & x \neq 0\\
0 & x = 0 
\end{cases}.
\]
We also define $\tilde{t}: \mathbb{R}^d \to \mathbb{R}$ in the same way. Then $\lim_{k \to \infty}\sup_{p \in \mathbb{S}^{d-1}}|s(p,A_{n_k}) - t(p)| = 0$ implies $\lim_{k \to \infty}\sup_{x:\|x\|\leq R}|\tilde{s}(x,A_{n_k}) - \tilde{t}(x)| = 0$ for any $R>0$. 

Now we verify (a) positive homogeneity and (b) convexity of $\tilde{t}(\cdot)$ that characterize support functions \citep{Sc14}. For (a), pick any $\lambda \geq 0$ and $x \in \mathbb{R}^d$. We have 
\begin{align*}
|\tilde{t}(\lambda x) - \lambda \tilde{t}(x)| \leq |\tilde{t}(\lambda x) - \tilde{s}(\lambda x, A_{n_k})| + |\tilde{s}(\lambda x, A_{n_k}) - \lambda\tilde{s}(x, A_{n_k})| + |\lambda \tilde{s}(x, A_{n_k}) - \lambda \tilde{t}(x)| \to 0,
\end{align*} 
as $k \to \infty$, which implies (a) as $\tilde{t}(\lambda x) = \lambda \tilde{t}(x)$. For (b), if $\tilde{t}$ is subadditive (i.e., $\tilde{t}(x+y) \leq \tilde{t}(x) + \tilde{t}(y)$ for any $x, y \in \mathbb{R}^d$), then from the positive homogeneity of $\tilde{t}$, we can see that $\tilde{t}$ is convex. Now we verify the subadditivity of $\tilde{t}$. For each $k$, $\tilde{s}(\cdot, A_{n_k})$ is subadditive. Then we have $0 \leq \tilde{s}(x, A_{n_k}) + \tilde{s}(y, A_{n_k})  - \tilde{s}(x+y, A_{n_k})$. This yields $0 \leq \tilde{t}(x) + \tilde{t}(y)  - \tilde{t}(x+y)$, which implies the subadditivity of $\tilde{t}$, and we obtain (b).

\noindent
\textbf{(Step 3)} Now we verify that $\tilde{t}$ is a support function of a bounded closed convex set $A$. Define $A = \bigcap_{p \in \mathbb{S}^{d-1}}\left\{x \in \mathbb{R}^d: \langle p, x \rangle \leq  \tilde{t}(p)\right\}$. Then we can see that $A$ is a closed convex set and $\tilde{t}(\cdot)$ is a support function of $A$. Moreover, for any $x \in A$ with $x \neq 0$, $\|x\|=\langle x/\|x\|, x \rangle \leq t(x/\|x\|) \leq \sup_n (\sup_{a \in A_n}\|a\|) \leq R$. Then $A$ is a bounded closed convex set. In particular $\tilde{t}(\cdot) = s(\cdot, A)$, $A \in K_{kc}^R(\mathbb{R}^d)$. 

\subsection{Proof of Proposition \ref{prp:F-mean-embed}}

\subsubsection{Proof of (i)}

$\Psi(Y) \in \Psi(\mathcal{M})$, which is a closed convex set (Assumption \ref{ass:1} (ii)). Then from Lemma \ref{lem:convex-hilbert-exp}, we have $\mathbb{E}[\Psi(Y)] \in \Psi(\mathcal{M})$. Since $\Psi$ is injective (Assumption \ref{ass:1} (i)) and $\mathbb{E}[\Psi(Y)] \in \Psi(\mathcal{M})$, $\Psi^{-1}(\mathbb{E}[\Psi(Y)])$ is well defined. From the definition of $\mathbb{E}[\Psi(Y)]$, we have $\mathbb{E}[\langle h, \Psi(Y) \rangle_\mathcal{H}] = \langle h, \mathbb{E}[\Psi(Y)] \rangle_\mathcal{H}$ for $h \in \mathcal{H}$. This yields 
\begin{align*}
\mathbb{E}[\Psi(Y)] 
&= \mathrm{argmin}_{h \in \mathcal{H}}\mathbb{E}[\|h - \Psi(Y)\|_\mathcal{H}^2]
= \mathrm{argmin}_{h \in \Psi(\mathcal{M})}\mathbb{E}[\|h - \Psi(Y)\|_\mathcal{H}^2]. 
\end{align*}
Thus, we have $\mathbb{E}[\|\mathbb{E}[\Psi(Y)] - \Psi(Y)\|_\mathcal{H}^2] 
\leq \mathbb{E}[\|\Psi(\nu) - \Psi(Y)\|_\mathcal{H}^2]$ for any $\nu \in \mathcal{M}$. By the isometry of $\Psi$ (Assumption \ref{ass:1} (i)), we have $\mathbb{E}\left[d^2\left(\Psi^{-1}(\mathbb{E}[\Psi(Y)]), Y\right)\right] \leq \mathbb{E}[d^2(\nu,Y)]$ for any $\nu \in \mathcal{M}$. Therefore, we obtain $\mathbb{E}_\oplus[Y] = \Psi^{-1}(\mathbb{E}[\Psi(Y)])$. 

\subsubsection{Proof of (ii)}

Applying the same argument in the proof of (i) to the conditional distribution of $Y$ given $X$, we have $\Psi(\mathbb{E}_\oplus[Y|X]) = \mathbb{E}[\Psi(Y)|X]$. Then we have $\mathbb{E}[\Psi(\mathbb{E}_\oplus[Y|X])] = \mathbb{E}[\Psi(Y)] = \Psi(\mathbb{E}_\oplus[Y])$, where the last equality follows from (i). 

\subsection{Proof of Proposition \ref{prp: SVRV-exp-cexp}}

\subsubsection{Proof of (i) and (ii)}

We show the result only for the Fr\'echet mean since a similar argument applies to the conditional Fr\'echet mean. $\Psi(K_{kc}^R(\mathbb{R}^d)) \subset L^2(\mathbb{S}^{d-1})$ is closed and convex (Lemma \ref{lem:d} (iv)).  Then from Proposition \ref{prp:F-mean-embed} (i), we have $s(\cdot, \mathbb{E}_\oplus[F])=\Psi(\mathbb{E}_\oplus[F])=\mathbb{E}[\Psi(F)] = \mathbb{E}[s(\cdot, F)]$. Since $F$ is an integrably bounded SVRV, Lemma \ref{lem:int-bdd-SVRV} yields $s(\cdot, \mathbb{E}_\oplus[F]) = s(\cdot, \mathbb{E}[F])$, where $\mathbb{E}[F]$ is the Aumann expectation of $F$. Then Lemma \ref{lem:sf-cc-unique} yields $\mathbb{E}_\oplus[F] = \mathbb{E}[F]$. 

\subsubsection{Proof of (iii)}

The first statement follows by applying a similar argument to the proof of (i). $\mathbb{E}[F|X] \in K_{kc}^R(\mathbb{R}^d)$ and $\mathbb{E}[F] = \mathbb{E}[\mathbb{E}[F|X]]$. Thus, from (i) and  Proposition \ref{prp:F-mean-embed} (ii), we have
\[
\Psi(\mathbb{E}_\oplus[\mathbb{E}_\oplus[F|X]]) =\mathbb{E}[\Psi(\mathbb{E}[F|X])] = s(\cdot, \mathbb{E}[\mathbb{E}[F|X]]) = s(\cdot, \mathbb{E}[F]) =\Psi(\mathbb{E}_\oplus[F]),
\]
which yields $\mathbb{E}_\oplus[F] = \mathbb{E}_\oplus[\mathbb{E}_\oplus[F|X]]$.

\subsection{Proof of Corollary \ref{cor:sample-F-A-exp}}
From Proposition \ref{prp:F-mean-embed}, we have $s(\cdot, \mu_{\oplus,n})=\Psi(\mu_{\oplus,n})={1 \over n}\sum_{i=1}^n\Psi(F_i) = {1 \over n}  \sum_{i=1}^n s(\cdot, F_i)$. Thus, Lemma \ref{lem:sf-add-conv} (iii) yields $s(\cdot, \mu_{\oplus,n}) = s\left(\cdot,  {1 \over n} \bigoplus_{i=1}^n F_i \right)$, and Lemma \ref{lem:sf-cc-unique} implies $\mu_{\oplus,n} = {1 \over n}  \bigoplus_{i=1}^n  F_i$. 

\subsection{Proof of Proposition \ref{prp:F-Reg-solution}}

First, we verify that $m_\oplus(x)$ is the GFR function. $\widetilde{m}_{\oplus,\Psi}(x)$ exists and is unique by Lemmas \ref{lem:d} and \ref{lem:A1}. From the definition of $\widetilde{m}_{\oplus,\Psi}(x)$, we have
\begin{align*}
\mathbb{E}[w(x,X)d_{kc}^2(\nu, F)]
&= \|\Psi(\nu) -\mathbb{E}[w(x,X)\Psi(F)]\|_{2,\mathbb{S}^{d-1}}^2\\
&\quad + \int_{\mathbb{S}^{d-1}}\mathbb{E}[w(x,X)s^2(p,F)]dp-\|\mathbb{E}[w(x,X)\Psi(F)]\|_{2,\mathbb{S}^{d-1}}^2\\
&\geq \|\widetilde{m}_{\oplus,\Psi}(x) - \mathbb{E}[w(x,X)\Psi(F)]\|_{2,\mathbb{S}^{d-1}}^2\\
&\quad + \int_{\mathbb{S}^{d-1}}\mathbb{E}[w(x,X)s^2(p,F)]dp-\|\mathbb{E}[w(x,X)\Psi(F)]\|_{2,\mathbb{S}^{d-1}}^2\\
&=\mathbb{E}[w(x,X)d_{kc}^2(\Psi^{-1}(\tilde{m}_{\oplus,\Psi}(x)), F)],
\end{align*}
for any $\nu \in K_{kc}^R(\mathbb{R}^d)$. This yields $\Psi^{-1}(\widetilde{m}_{\oplus,\Psi}(x)) \in \mathrm{argmin}_{\nu \in K_{kc}^R(\mathbb{R}^d)}\mathbb{E}[w(x,X)d_{kc}^2(\nu,F)]$ so that $\Psi^{-1}(\widetilde{m}_{\oplus,\Psi}(x))$ is the unique GFR function. 

Next, we verify $m_\oplus(x) \subset \mathbb{E}[w(x,X)F]$. $w(x,X)F \in K_{kc}^{B'}(\mathbb{R}^d)$ where $B'=B\sup_{z \in \mathcal{X}}|w(x,z)|$. Then Lemma \ref{lem:d} (iv), Proposition \ref{prp:F-mean-embed} (i), and Proposition \ref{prp: SVRV-exp-cexp} (ii) yield $\mathbb{E}[w(x,X)F] = \mathbb{E}_\oplus[w(x,X)F] \in K_{kc}^{B'}(\mathbb{R}^d) \subset K_{kc}(\mathbb{R}^d)$. Since $m_{\oplus,\Psi}(x) \in \Psi(K_{kc}(\mathbb{R}^d))$, $m_{\oplus,\Psi}(x) = \widetilde{m}_{\oplus,\Psi}(x)$ is the support function of $m_\oplus(x)=\Psi^{-1}(m_{\oplus,\Psi}(x)) \in K_{kc}(\mathbb{R}^d)$ and hence (\ref{ineq:GFR-BLP}) yields $m_{\oplus,\Psi}(x)=s(\cdot,m_\oplus(x))\leq s(\cdot, \mathbb{E}[w(x,X)F])$. Therefore, from Lemma \ref{lem:sf-ineq}, we have $m_\oplus(x) \subset \mathbb{E}[w(x,X)F]$. 

\subsection{An example such that $\tilde{m}_{\oplus,\Psi}(x) < \Psi(\mathbb{E}[w(x,X)F])$}\label{subsec:app-counter-exm}
Let $(F,X) \in \{[-1,2], [1,6], \{0\}\} \times \{-2,0,2\}$ and let
\begin{align*}
&\mathbb{P}(F=[-1,2], X=-2)=\mathbb{P}(F=[1,6], X=2)=1/4,\\
&\mathbb{P}(F \neq [-1,2], X=-2)=\mathbb{P}(F \neq [1,6], X=2)=0,\\
&\mathbb{P}(F = \{0\},X=0)=1/2,\quad \mathbb{P}(F \neq \{0\}, X=0)=0.
\end{align*}
$\mathbb{S}^0=\{-1,1\}$. Define $g_1(p)=\mathbb{E}[w(2,X)s(p,F)]$ and $g_2(p)=s(p, \mathbb{E}[w(2,X)F])$. Then we have $g_1(p)< g_2(p)$ for $p \in \mathbb{S}^0$. Indeed,
\begin{align*}
g_1(p) & = -{1 \over 4}s(p,[-1,2]) + {3 \over 4}s(p,[1,6])=
\begin{cases}
4 & \text{if}\ p=1\\ 
-1 & \text{if}\ p=-1
\end{cases},\\
g_2(p) &= s\left(p, -{1 \over 4}[-1,2]\oplus {3 \over 4}[1,6]\right)=s\left(p,\left[{1 \over 4},{19 \over 4}\right]\right)=
\begin{cases}
{19 \over 4} & \text{if}\ p=1\\
-{1 \over 4} & \text{if}\ p=-1
\end{cases}.
\end{align*}
$\|g\|_{2,\mathbb{S}^0}^2=g^2(-1) + g^2(1)$. Then we have $0 = \|g_1-s(\cdot,[1,4])\|_{2,\mathbb{S}^0}^2 <\|g_1 - g_2\|_{2,\mathbb{S}^0}^2=18/16$. This implies $s(\cdot,[1,4])=\tilde{m}_{\oplus,\Psi}(2) \neq \Psi(\mathbb{E}[w(2,X)F])=s(\cdot,[1/4,19/4])$. Likewise, we have $\tilde{m}_{\oplus,\Psi}(0) = \Psi(\mathbb{E}[w(0,X)F])=s(\cdot, \mathbb{E}[F])=s(\cdot,[0,2])$ and $s(\cdot,[-1,0])=\tilde{m}_{\oplus,\Psi}(-2) \neq \Psi(\mathbb{E}[w(-2,X)F])=s(\cdot,[-9/4,5/4])$.

\subsection{Proof of Proposition \ref{prp:GFR-rate}}

We first show the result for $d\geq 2$ and then show the result for $d=1$. 

\subsubsection{Proof for the case $d \geq 2$}
Define 
\[
Q(\nu,x)=\mathbb{E}[w(x,X)d_{kc}^2(\nu,F)],\quad  Q_n(\nu,x)={1 \over n}\sum_{i=1}^n\hat{w}(x,X_i)d_{kc}^2(\nu,F_i).
\]
The proof follows the same empirical process argument as Theorem A.4 of \cite{KuZhOtMu24}, adapting it to the present global Fréchet regression setting. In particular, it suffices to verify the following conditions.
\begin{itemize}
\item[(a)] Almost surely, for all $\|x\| \leq B_0$, the objects $m_\oplus(x)$ and $\hat{m}_\oplus(x)$ exist and are unique. Additionally, for any $\varepsilon>0$, 
\[
\inf_{\|x\|\leq B_0}\inf_{d_{kc}(\nu,m_\oplus(x))>\varepsilon}Q(\nu,x) - Q(m_\oplus(x),x)>0,
\]
and there exists $\zeta = \zeta(\varepsilon)>0$ such that
\[
\lim_{n \to \infty}\mathbb{P}\left(\inf_{\|x\|\leq B_0}\inf_{d_{kc}(\nu,\hat{m}_\oplus(x))>\varepsilon}Q_n(\nu,x) - Q_n(\hat{m}_\oplus(x),x) \geq \zeta\right) = 1.
\]
\item[(b)] Let $B_\delta(m_\oplus(x)) \subset K_{kc}^R(\mathbb{R}^d)$ be the ball of radius $\delta$ centered at $m_\oplus(x)$ and $N(\varepsilon, B_\delta(m_\oplus(x)),d_{kc})$ be its covering number using ball of size $\varepsilon$. As $\delta \to 0$, it holds
\[
\int_0^1 \sup_{\|x\|\leq B_0}\sqrt{1 + \log N(\delta\varepsilon, B_\delta(m_\oplus(x)),d_{kc})}d\varepsilon = O(\delta^{-{d-1 \over 4}}). 
\]
\item[(c)] There exist $\tau>0$ and $D>0$, possibly depending on $B_0$, such that
\[
\inf_{\|x\|\leq B_0}\inf_{d_{kc}(\nu,m_\oplus(x))<\tau}\{Q(\nu,x) - Q(m_\oplus(x),x) - Dd_{kc}^2(\nu,m_\oplus(x))\} \geq 0. 
\]
\end{itemize}

\subsubsection*{Verification of (a)}

Applying the same argument in the proof of Proposition \ref{prp:F-Reg-solution}, we can see that $m_\oplus(x)$ and $\hat{m}_\oplus(x)$ exist and are unique for all $\|x\| \leq B_0$, hence proving Condition (a). 

\subsubsection*{Verification of (b)}

By Lemma~\ref{lem:covering-number}, for $d\ge2$,
\[
\sup_{\|x\|\le B_0} \log N\!\left( \delta\varepsilon, B_\delta(m_\oplus(x)), d_{kc} \right) \le D(\delta\varepsilon)^{-(d-1)/2}.
\]
Therefore,
\[
\int_0^1 \sup_{\|x\|\le B_0} \sqrt{1+\log N(\delta\varepsilon, B_\delta(m_\oplus(x)),d_{kc})}\,d\varepsilon \le \int_0^1 \sqrt{1+D(\delta\varepsilon)^{-(d-1)/2}}\,d\varepsilon = O\!\left(\delta^{-(d-1)/4}\right).
\]

\subsubsection*{Verification of (c)}

For $f,g \in L^2(\mathbb{S}^{d-1})$, let $\langle f,g\rangle_{2,\mathbb{S}^{d-1}} =\int_{\mathbb{S}^{d-1}}f(p)g(p)dp$. The role of $w(x,X)$ can be seen by expanding the weighted objective. Since $\mathbb{E}[w(x,X)]=1$, its dependence on $\nu$ is, up to terms independent of $\nu$, $\|\Psi(\nu)-m_{\oplus,\Psi}(x)\|_{2,\mathbb{S}^{d-1}}^2$ and $m_{\oplus,\Psi}(x)=\mathbb{E}[w(x,X)\Psi(F)]$. Thus, $w(x,X)$ determines the unconstrained Hilbert-space target $m_{\oplus,\Psi}(x)$, while $\widetilde{m}_{\oplus,\Psi}(x)=\Psi(m_\oplus(x))$ is its projection onto $\Psi(K_{kc}^R(\mathbb{R}^d))$. The inner-product term below therefore reflects the discrepancy between the weighted unconstrained target and its projection. By (\ref{eq:mop}), $\widetilde m_{\oplus,\Psi}(x)$ is the projection of $m_{\oplus,\Psi}(x)$ onto $\Psi(K_{kc}^R(\mathbb R^d))$. From Lemma \ref{lem:inn-prod}, $m_\oplus(x)$ is characterized by $\langle m_{\oplus,\Psi}(x)- \tilde{m}_{\oplus,\Psi}(x), h - \tilde{m}_{\oplus,\Psi}(x)\rangle_{2,\mathbb{S}^{d-1}} \leq 0$ for all $h \in \Psi(K_{kc}^R(\mathbb{R}^d))$ and $\|x\| \leq B_0$. Then we have 
\begin{align*}
Q(\nu,x) - Q(m_\oplus(x),x)
&= d_{kc}^2(\nu, m_\oplus(x)) - 2\langle m_{\oplus,\Psi}(x)- \tilde{m}_{\oplus,\Psi}(x), \Psi(\nu) - \tilde{m}_{\oplus,\Psi}(x)\rangle_{2,\mathbb{S}^{d-1}}\\
&\geq d_{kc}^2(\nu, m_\oplus(x)).
\end{align*}
Consequently, we can take $D=1$ and $\tau$ arbitrary to satisfy Condition (c). 

\subsubsection{Proof for the case $d=1$}

Verification of (a) and (c) is the same as in $d \geq 2$. One can see that $N(\varepsilon, K_{kc}^R(\mathbb{R}),d_{kc})
\leq D\varepsilon^{-2}$ where $D$ is independent of $\varepsilon$. By applying a similar argument to verify (b) when $d \geq 2$, we have 
\[
\int_0^1 \sup_{\|x\|\leq B_0}\sqrt{1 + \log N(\delta\varepsilon, B_\delta(m_\oplus(x)),d_{kc})}d\varepsilon = O(-\log \delta).
\]
Applying the same empirical process argument as in the proof of Theorem A.4 of \cite{KuZhOtMu24}, we obtain $\sup_{\|x\|\leq B_0}d_{kc}(\hat{m}_\oplus(x),m_\oplus(x)) = O_p(n^{-{1 \over 2(\alpha-1)}})$ for any $\alpha>2$.

\subsection{An example (non-identification of $\Theta$ and non-sharpness of $\tilde\Theta$ in the errors-in-variable model)}\label{ex:nonsharp}
Consider the errors-in-variable model: $W=X+\varepsilon$ and $Z=X+v$ with the constant interval-valued outcome $F=[-1,1]$.
Conditional on the latent variable $X$, suppose that $W$ and $Z$ are independent and take values in $\{-1,1\}$, with $\mathbb{P}(W=1\mid X) = \mathbb{P}(Z=1\mid X) = (1+X)/2$. This construction satisfies $\mathbb{E}[\varepsilon\mid X,F] = \mathbb{E}[v\mid X,F] = \mathbb{E}[\varepsilon v\mid X,F]=0$.
To see the non-identification, consider the following two data-generating processes (DGPs):
\begin{itemize}
    \item[A)] $\mathbb{P}(X=1/2) = \mathbb{P}(X=-1/2) = 1/2$.
    \item[B)] $\mathbb{P}(X=1) = \mathbb{P}(X=-1) = 1/8$ and $\mathbb{P}(X=0) = 3/4$.
\end{itemize}
In both cases, $\mathbb E[X] = 0$ and $\mathbb E[X^2] = 1/4$.
We therefore obtain that
\begin{align*}
    \mathbb P(W=w, Z=z) &=
    \mathbb E[\mathbb P(W=w, Z=z \mid X)] = 
    \mathbb E[\mathbb P(W=w \mid X) \mathbb P(Z=z \mid X)]\\
    &= \mathbb E\left[\frac{(1+wX)(1+zX)}{4}\right] 
    = \frac{1}{4} \left( 1 + (w+z)\mathbb E[X] wz\mathbb E[X^2]\right)
    = \frac{1}{4}\left(1 + \frac{wz}{4}\right)
\end{align*}
for both DGPs. Since $F=[-1,1]$ is constant, the two DGPs are observationally equivalent. We note that $\mathbb V[X] = 1/4 = \mathbb{C}\mathrm{ov}[W,Z]$.

The original Beresteanu--Molinari coefficient set is
\begin{align*}
    \Theta= \left\{\begin{pmatrix}
        \mathbb E[f]\\ 4\mathbb E[Xf] 
    \end{pmatrix} : f\in \mathcal S^1(F)\right\}
    =
    \left\{\begin{pmatrix}
        \mathbb E[f]\\ 4\mathbb E[Xf] 
    \end{pmatrix} : |f|\le 1\right\}.
\end{align*}
Under DGP A, a direct calculation gives
\begin{align*}
    \Theta_A = \left\{(\theta_0,\theta_1)' \in \mathbb R^2 : 
    |\theta_0| + \frac{|\theta_1|}{2} \le 1\right\}.
\end{align*}
Under DGP B, we instead obtain
\begin{align*}
    \Theta_B = \left\{(\theta_0,\theta_1)' \in \mathbb R^2 : 
    |\theta_1| \le 1,\,\,|\theta_0| + \frac{|\theta_1|}{4} \le 1\right\}.
\end{align*}
Therefore $\Theta_A \neq \Theta_B$, despite the identical joint distribution of $(W,Z,F)$. This shows that the original Beresteanu–Molinari coefficient set is not generally identified under the maintained assumptions.

In the main text, we have shown that $\Theta \subseteq \tilde\Theta$. In the present example, we can compute that
\begin{align*}
    \tilde\Theta=
    \left\{\begin{pmatrix}
        \mathbb E[f]\\ 4\mathbb E[Zf] 
    \end{pmatrix} : |f|\le 1\right\}
    =\left\{(\theta_0,\theta_1)' \in \mathbb R^2 : 
    |\theta_0| + \frac{|\theta_1|}{4} \le 1\right\},
\end{align*}
which indeed satisfies $\Theta_A \subseteq \tilde\Theta$ and $\Theta_B \subseteq \tilde\Theta$.

However, we can show that $\tilde{\Theta}$ is not necessarily a sharp set for the coefficient vector.
By choosing $f=Z$, which is a valid selection of $F=[-1,1]$, we have $\mathbb E[f]=0$ and $4\mathbb E[Zf] = 4$. Hence $(0,4)' \in \tilde\Theta$. 
On the other hand, the Cauchy-Schwarz inequality implies $4\mathbb E[Xf] \le 4\sqrt{\mathbb E[X^2] \mathbb E[f]} \le 2$. Therefore, $(0,4)'$ cannot belong to $\Theta$ under any DGP consistent with the observed distribution and the maintained assumptions. Consequently, $\tilde\Theta$ is a valid outer bound, but this bound need not be sharp.

\subsection{Derivation of (\ref{eq:IPW})}\label{subsec:app-IPW}

For $\alpha, \beta \in K_{kc}^R(\mathbb{R}^d)$, let $\gamma_{\alpha,\beta}:[0,1] \to K_{kc}^R(\mathbb{R}^d)$ denote the geodesic from $\alpha$ to $\beta$ under the metric $d_{kc}$, and $\gamma_{\alpha,\beta}(\rho)$ be the end point of the (extended) geodesic $\rho \odot \gamma_{\alpha,\beta}$. See \cite{KuZhOtMu24} for a more detailed discussion on geodesic metric spaces and extensions of geodesics. Since $K_{kc}^R(\mathbb R^d)$ can be isometrically embedded into $L^2(\mathbb S^{d-1})$ via $\Psi(A)=s(\cdot,A)$, we have
\[
\gamma_{\{0\},F_i}\left({T_i \over \hat{e}(X_i)}\right) = \Psi^{-1}\left(\Psi(\{0\})+{T_i \over \hat{e}(X_i)}(\Psi(F_i) -\Psi(\{0\}))\right)
= \Psi^{-1}\left({T_i \over \hat{e}(X_i)}\Psi(F_i)\right)={T_i \over \hat{e}(X_i)}F_i.
\]
Therefore, we have
\begin{align*}
\hat{\Theta}_1^{(\mathrm{IPW})}=\mathrm{argmin}_{\nu \in K_{kc}^R(\mathbb{R}^d)}{1 \over n}\sum_{i=1}^nd_{kc}^2\left(\nu,\gamma_{\{0\},F_i}\left({T_i \over \hat{e}(X_i)}\right)\right) ={1 \over n} \bigoplus_{i=1}^n {T_i \over \hat{e}(X_i)} F_i. 
\end{align*} 

\subsection{Proof of Proposition \ref{prp:IPW-conv}}
\begin{align*}
&d_{kc}(\hat{\mathbb{E}}_\oplus[F],\mathbb{E}_\oplus[F])\\ 
&= \left\|s\left(\cdot, \left({1 \over n}\sum_{i=1}^n {T_i \over \hat{e}(X_i)}\right)^{-1}{1 \over n}\bigoplus_{i=1}^n {T_i \over \hat{e}(X_i)}F_i\right) - s(\cdot, \mathbb{E}_\oplus[F])\right\|_{2,\mathbb{S}^{d-1}}\\
&\leq {B \over \eta_0}\sqrt{\mathrm{vol}(\mathbb{S}^{d-1})}\left|\left({1 \over n}\sum_{i=1}^n {T_i \over \hat{e}(X_i)}\right)^{-1}-1\right| + \left\|s\left(\cdot, {1 \over n}\bigoplus_{i=1}^n {T_i \over \hat{e}(X_i)}F_i\right) - s(\cdot, \mathbb{E}_\oplus[F])\right\|_{2,\mathbb{S}^{d-1}}\\
&\leq {B \over \eta_0}\sqrt{\mathrm{vol}(\mathbb{S}^{d-1})}\left\{{1 \over \eta_0^2}\sup_{x \in \mathcal{X}}|\hat{e}(x) - e(x)| + \left|{1 \over n}\sum_{i=1}^n {T_i \over e(X_i)}-1\right|\right\}(1 + o_p(1)) + d_{kc}(\hat{\Theta}_1^{(\mathrm{IPW})}, \mathbb{E}_\oplus[F])\\
&= O_p(\rho_n) + O_p(n^{-1/2}) + d_{kc}(\hat{\Theta}_1^{(\mathrm{IPW})}, \mathbb{E}_\oplus[F]). 
\end{align*}
Therefore, the conclusion follows from a similar argument in the proof of Theorem A.4 in \cite{KuZhOtMu24}. 

\section{Auxiliary results}\label{app:B}

This appendix collects auxiliary results used in the proofs in Appendix \ref{app:A}. We first present some results on Hilbert space.

\begin{lemma}\label{lem:convex-hilbert-exp}
If $C$ is a nonempty closed convex subset of a Hilbert space $\mathcal{H}$ and $Z$ is a random element taking values in $C$ with $\mathbb{E}[\|Z\|_\mathcal{H}]<\infty$, then $\mathbb{E}[Z] \in C$.
\end{lemma}

\begin{proof}
Let $\mu_Z = \mathbb{E}[Z]$ be the expectation of $Z$ and suppose $\mu_Z \notin C$. Since $C$ is closed, there is an open ball $B_\delta(\mu_Z) = \{h \in \mathcal{H}: \|h - \mu_Z\|_\mathcal{H}<\delta\}$ such that $B_\delta(\mu_Z) \cap C = \emptyset$. By Hahn-Banach separation theorem, 
there exists an $h_0 \in \mathcal{H}$ and $a \in \mathbb{R}$ such that $\langle h_0, \mu_Z \rangle_\mathcal{H} < a \leq \langle h_0, h \rangle_\mathcal{H}$ for all $h \in C$. Since $a \le \langle h_0,h\rangle_\mathcal{H}$ for all $h\in C$, we have
\[
a \le \mathbb E[\langle h_0,Z\rangle_\mathcal{H}] = \langle \mathbb E[Z],h_0\rangle_\mathcal{H} = \langle \mu_Z,h_0\rangle_\mathcal{H},
\]
where the equality follows from the linearity of the expectation and the inner product. This contradicts the strict separation $\langle h_0,\mu_Z\rangle_\mathcal{H}<a$.
\end{proof}

\begin{lemma}\label{lem:A1}
Let $C$ be a nonempty closed convex set of a Hilbert space $\mathcal{H}$. For any $y \in \mathcal{H}$, the projection $\pi_C(y) = \mathrm{argmin}_{h \in C} \|y-h\|_\mathcal{H}$ exists and is unique.
\end{lemma}

\begin{proof}
This is the classical metric projection theorem for Hilbert spaces; see Theorem~2.5 in \cite{Conway90}.
\end{proof}

\begin{lemma}\label{lem:inn-prod}
For a Hilbert space $\mathcal{H}$, let $C\subset\mathcal{H}$ be a nonempty closed convex set and $y \in \mathcal{H}$. Then for every $h\in C$, $\langle y-\pi_C(y),\, h-\pi_C(y)\rangle_\mathcal{H} \leq 0$.
\end{lemma}

\begin{proof}
This is the standard variational characterization of the metric projection; see, for example, the optimality condition in Theorem~27.4 of \cite{Rockafellar70}.
\end{proof}

Hereafter, we summarize some results on compact convex sets and their support functions used in the main text.

\begin{lemma}\label{lem:sf-conti}
If $A \subset \mathbb{R}^d$ is a  bounded set (i.e., $\sup_{a \in A}\|a\|<\infty$), then the support function $s(p, A)$, $p \in \mathbb{S}^{d-1}$ is Lipschitz continuous and uniformly continuous on $\mathbb{S}^{d-1}$. 
\end{lemma}
\begin{proof}
For any $u,v \in \mathbb{R}^d$ and any $a \in A$, we have $\langle u, a \rangle - \langle v, a \rangle \leq \langle u-v, a \rangle  \leq (\sup_{a \in A}\|a\|)\|u - v\|$. Then we have $s(u,A) - s(v,A)  \leq (\sup_{a \in A}\|a\|)\|u - v\|$. Likewise, we have $s(v,A) - s(u,A) \leq (\sup_{a \in A}\|a\|)\|u - v\|$, which implies $|s(u,A) - s(v,A)| \leq (\sup_{a \in A}\|a\|)\|u - v\|$, and the conclusion follows.
\end{proof}

\begin{lemma}\label{lem:cc-set-rep}
Let $s(\cdot, A)$ be the support function of $A\in K_{kc}(\mathbb{R}^d)$. Then we have $A = \bigcap_{p \in \mathbb{S}^{d-1}}\{x \in \mathbb{R}^d: \langle p, x \rangle \leq s(p, A)\}= F_S$. 
\end{lemma}
\begin{proof}
This follows from Corollary~13.1.1 in \cite{Rockafellar70} together with the positive homogeneity of the support function.
\end{proof}

\begin{lemma}\label{lem:sf-add-conv}
The following results hold true. 
\begin{itemize}
\item[(i)] (Positive homogeneity) For any $p \in \mathbb{R}^d$, $t \geq 0$, and subset $A \subset \mathbb{R}^d$,  
\[
s(tp, A) = ts(p,A),\quad s(p, t A) = ts(p,A). 
\] 
\item[(ii)] (Additivity) For any $p \in \mathbb{R}^d$ and subsets $A, B \subset \mathbb{R}^d$, 
\[
s(p, A \oplus B) = s(p, A) + s(p, B).
\] 
\item[(iii)] (Convexity) For any $p \in \mathbb{R}^d$, $t \in [0,1]$, and subsets $A, B \subset \mathbb{R}^d$,
\[
s(p, (t A) \oplus ((1-t) B)) = ts(p,A) + (1-t)s(p,B).
\] 
\end{itemize}
\end{lemma}
\begin{proof}
Parts (i) and (ii) are standard properties of support functions; see Section~13, in particular Theorem~13.2 and the discussion following Corollary~13.1.1, in \cite{Rockafellar70}. Part (iii) follows immediately from parts (i) and (ii).
\end{proof}

\begin{lemma}\label{lem:sf-ineq}
For $A,B \in K_{kc}(\mathbb{R}^d)$, $s(p,A)\leq s(p,B)$ for all $p \in \mathbb{S}^{d-1}$ if and only if $A \subset B$. 
\end{lemma}
\begin{proof}
This follows from Corollary~13.1.1 in \cite{Rockafellar70} together with the positive homogeneity of the support function.
\end{proof}

\begin{lemma}\label{lem:sf-cc-unique}
Let $s(\cdot, A)$ and $s(\cdot, B)$ be the support functions of $A,B \in K_{kc}(\mathbb{R}^d)$, respectively. Then  $A = B$ if and only if $s(p, A)=s(p, B)$ for $p \in \mathbb{S}^{d-1}$. 
\end{lemma}
\begin{proof}
This follows immediately from Lemma~\ref{lem:sf-ineq}.
\end{proof}

\begin{lemma}[Lemma 1.8.14 in \cite{Sc14}]\label{lem:Hd-sf-conv}
For any $A,B \in K_{kc}(\mathbb{R}^d)$, we have $d_H(A, B) = \sup_{p \in \mathbb{S}^{d-1}}|s(p,A) - s(p,B)|$. 
\end{lemma}

\begin{lemma}[Theorem 3 in \cite{Vi85}]\label{lem:bcc-set-conv}
Let $\{A_n\}_{n \geq 1} \subset  K_{kc}(\mathbb{R}^d)$ be a sequence of compact convex sets. If $\lim_{n \to \infty}d_H(A_n, A) =0$, then $A \subset \mathbb{R}^d$ is a compact convex set.
\end{lemma}

Lemmas \ref{lem:d} (ii) and \ref{lem:bcc-set-conv} imply that $(K_{kc}(\mathbb{R}^d),d_{kc})$ (or $(K_{kc}(\mathbb{R}^d), d_H)$) is a complete metric space. 

\begin{lemma}[Lemma 2.2 in \cite{Ar74}]\label{lem:int-bdd-SVRV}
If an SVRV $F: \Omega \to K_{kc}(\mathbb{R}^d)$ is integrably bounded, then we have $\mathbb{E}[s(p, F)] = s(p, \mathbb{E}[F])$ for $p \in \mathbb{S}^{d-1}$. 
\end{lemma}

\begin{lemma}\label{lem:covering-number}
Let $R>0$ and $N(\varepsilon, \mathcal{A},d)$ be the covering number of $\mathcal{A}$ with radius $\varepsilon$ under the metric $d$.
\begin{itemize}
\item[(i)] If $d=1$, there exists a constant $D>0$, depending only on $R$ such that
\[
N\bigl(\varepsilon,K_{kc}^R(\mathbb{R}),d_{kc}\bigr) \leq D\varepsilon^{-2}\quad \mbox{for all sufficiently small }\varepsilon>0.
\]
\item[(ii)] If $d\geq2$, there exists a constant $D>0$, depending only on $R$ and $d$ such that
\[
\log N\bigl(\varepsilon,K_{kc}^R(\mathbb{R}^d),d_{kc}\bigr) \leq D\varepsilon^{-(d-1)/2} \quad \mbox{for all sufficiently small }\varepsilon>0.
\]

\item[(iii)] For any $A\in K_{kc}^R(\mathbb{R}^d)$, $\delta\in(0,1)$, and $\varepsilon\in(0,1)$, $N\bigl(\delta\varepsilon,B_\delta(A),d_{kc}\bigr) \leq N\bigl(\delta\varepsilon,K_{kc}^R(\mathbb{R}^d),d_{kc}\bigr).$ Consequently, if $d\geq2$,
\[
\sup_{A\in K_{kc}^R(\mathbb{R}^d)}
\log N\bigl(\delta\varepsilon,B_\delta(A),d_{kc}\bigr)
\leq
D(\delta\varepsilon)^{-(d-1)/2}.
\]
\end{itemize}
\end{lemma}

\begin{proof}
For part (i), every element of $K_{kc}^R(\mathbb{R})$ can be written as $[a,b]$ with $-R\leq a\leq b\leq R$. Moreover, under the identification $[a,b]\mapsto(-a,b)$, the metric $d_{kc}$ is the Euclidean metric on a bounded subset of $\mathbb{R}^2$. The covering-number bound in (i) therefore follows from the standard covering-number bound for bounded subsets of $\mathbb{R}^2$. For part (ii), it follows from Lemma~2.7.12 of \cite{vaWe23}. Finally, part (iii) follows immediately from $B_\delta(A)\subset K_{kc}^R(\mathbb{R}^d)$.
\end{proof}

\bibliography{Frechet-RS}
\bibliographystyle{apalike}

\end{document}